\documentclass[12pt]{amsart}

\usepackage{amssymb}
\usepackage{hyperref}
\usepackage[T1]{fontenc}
\usepackage{bbm}
\usepackage{graphicx}
\usepackage{latexsym}
\usepackage{enumerate}
\usepackage{mathrsfs}
\usepackage{epstopdf}
\usepackage{color}
\usepackage{geometry}
\newtheorem{theorem}{Theorem}[section]

\newtheorem{lemma}[theorem]{Lemma}

\theoremstyle{definition}
\newtheorem{definition}[theorem]{Definition}

\theoremstyle{remark}
\newtheorem{remark}[theorem]{Remark}

\newcommand\R{\mathbb{R}}
\newcommand\e{\varepsilon}

\numberwithin{equation}{section}

\makeatletter
\@namedef{subjclassname@2020}{\textup{2020} Mathematics Subject Classification}
\makeatother%%%%%%%%%%MSC2020

\begin{document}

\author{Tomasz Ga\l azka}
\address{Faculty of Mathematics, Informatics and Mechanics\\
 University of Warsaw\\
Banacha 2, 02-097 Warsaw\\
Poland}
\email{T.Galazka@mimuw.edu.pl}

\author{Yong Jiao}
\address{School of Mathematics and Statistics, Central South University, Changsha 410085,
People's Republic of China}
\email{jiaoyong@csu.edu.cn}

\author{Adam Os\k ekowski}
\address{Faculty of Mathematics, Informatics and Mechanics\\
 University of Warsaw\\
Banacha 2, 02-097 Warsaw\\
Poland}
\email{A.Osekowski@mimuw.edu.pl}

\author{Lian Wu}
\address{School of Mathematics and Statistics, Central South University, Changsha 410085,
People's Republic of China}
\email{wulian@csu.edu.cn}

\thanks{This paper was supported by the National Key R$\&$D Program of China (No.2023YFA1010800), the NSFC (Nos.12125109, W2411005, 12361131578), Narodowe Centrum Nauki  Poland  (grant 2023/48/Q/ST1/00013), the Provincial Natural Science Foundation of Hunan (No.2024JJ1010) and
the Science and Technology Innovation Program of Hunan Province (No.2024RC1016).}


\title[Multilinear maximal operator]{Sharp weighted estimates for the multilinear maximal operator}

\begin{abstract}
The paper contains the study of the weighted $L^{p_1}\times L^{p_2}\times \ldots \times L^{p_m}\to L^p$ estimates for the multilinear maximal operator, in the context of abstract probability spaces equipped with a tree-like structure. Using the Bellman function method, we identify the associated optimal constants for weights satisfying the multilinear Muckenhoupt condition.
\end{abstract}

\maketitle

\section{Introduction}
The pioneering works of Calder\'on and Zygmund from the early 50's form the foundation for the large part of the contemporary harmonic analysis: the Calder\'on-Zygmund theory of singular integral operators. The initial motivation for the seminal paper \cite{CZ} came from connections 
 with potential theory and elliptic partial differential equations, as well as from questions about higher-dimension counterparts of the classical Hilbert transform. It was soon realized that the connections of the theory go much further and enable an efficient study of many problems in both real and complex analysis, operator theory, approximation theory and partial differential equations. Furthermore, the methods developed in the classical setting can be extended to various geometrical and multivariable contexts, including 
 homogeneous/non-homogeneous spaces, noncommutative theory and multilinear settings. We refer the interested reader to \cite{C,CM,DLMV,F, GT,MP,S,V} for more on the subject.
 
One of the basic questions in the Calder\'on-Zygmund theory concerns boundedness properties of these operators in various function spaces. This problem can be tackled with the use of various methods, one of which rests on the identification of appropriate maximal functions controlling the behavior of singular operators. This, in turn, leads to the question of efficient handling of such maximal functions, which in general is  more tractable and in many contexts allows for the identification of optimal, or at least tight constants involved. In this paper, we will be interested in the multilinear context \cite{CM,CM2,CJ,KS,GT0,GT1} and our primary goal is to establish a sharp weighted maximal $L^p$ bound arising in this setting. To proceed with the presentation, we need to introduce the necessary background and notation.

\smallskip

\emph{Linear case.} Suppose that $N$ is a fixed dimension and let $\mathscr{D}$ denote the standard dyadic lattice in $\R^N$. The dyadic maximal operator $M$ acts on locally integrable functions $f$ on $\R^N$  by the formula
$$ Mf=\sup_{Q\in \mathscr{D}} \langle |f|\rangle_Q\chi_Q,$$
where $\langle \psi\rangle_Q=\frac{1}{|Q|}\int_Q \psi$ is the average of $\psi$ over $Q$ with respect to the Lebesgue measure. This operator satisfies the $L^p$ bound
\begin{equation}\label{HLD}
 \|Mf\|_{L^p}\leq p'\|f\|_{L^p},\qquad 1<p\leq \infty,
\end{equation}
in which the constant $p'=p/(p-1)$ is the best possible. This estimate can be studied in the wider weighted context, where the word `weight' refers to a positive and locally integrable function on $\R^N$. Namely, any weight $w$ gives rise to the associated measure $w\mbox{d}x$ and the corresponding  weighted space $L^p(w)$. 
%$$ L^p(w)=\left\{f:\R^N\to \R\,:\,\|f\|_{L^p(w)}:=\left(\int_{\R^N}|f|^pw\mbox{d}x\right)^{1/p}<\infty\right\},\quad 0<p<\infty.$$
Now, there is an interesting problem of characterizing, for a fixed $1<p<\infty$, those weights $w$, for which $M$ acts boundedly on $L^p(w)$. The characterization was given by Muckenhoupt \cite{Mu} in the seventies: the boundedness holds if and only if $w$ belongs to the $A_p$ class, i.e., the $A_p$ characteristic of $w$, given by 
$$ [w]_{A_p}:=\sup_{Q\in \mathscr{D}} \langle w\rangle_Q \langle \sigma\rangle_Q^{p-1},$$
is finite. Here $\sigma=w^{1-p'}$ stands for the dual weight to $w$. The following quantitative version of this result was studied by Buckley in the nineties: for a given $1<p<\infty$, find the least exponent $\beta(p)$ such that
\begin{equation}\label{Buckley}
\|M\|_{L^p(w)\to L^p(w)}\leq \kappa_p [w]_{A_p}^{\beta(p)},
\end{equation}
where $\kappa_p$ depends only on $p$. That is, the problem is to identify the optimal dependence of $\|M\|_{L^p(w)\to L^p(w)}$ on the $A_p$ characteristic $[w]_{A_p}$. Buckley \cite{Bu} proved that the optimal exponent is equal to $\beta(p)=(p-1)^{-1}$, and this statement was further improved by the third-named author in \cite{Os}. To formulate the result, we need some additional notation. For given $1<p<\infty$ and $c\geq 1$, let $d=d(p,c)$ be the unique positive root of the equation
\begin{equation}\label{defd}
c(1+d)(p-1-d)^{p-1}=(p-1)^{p-1}.
\end{equation}
For the geometric interpretation of this parameter, see Figure \ref{interd} below.
\begin{figure}[htbp]
\begin{center}
\includegraphics[scale=0.4]{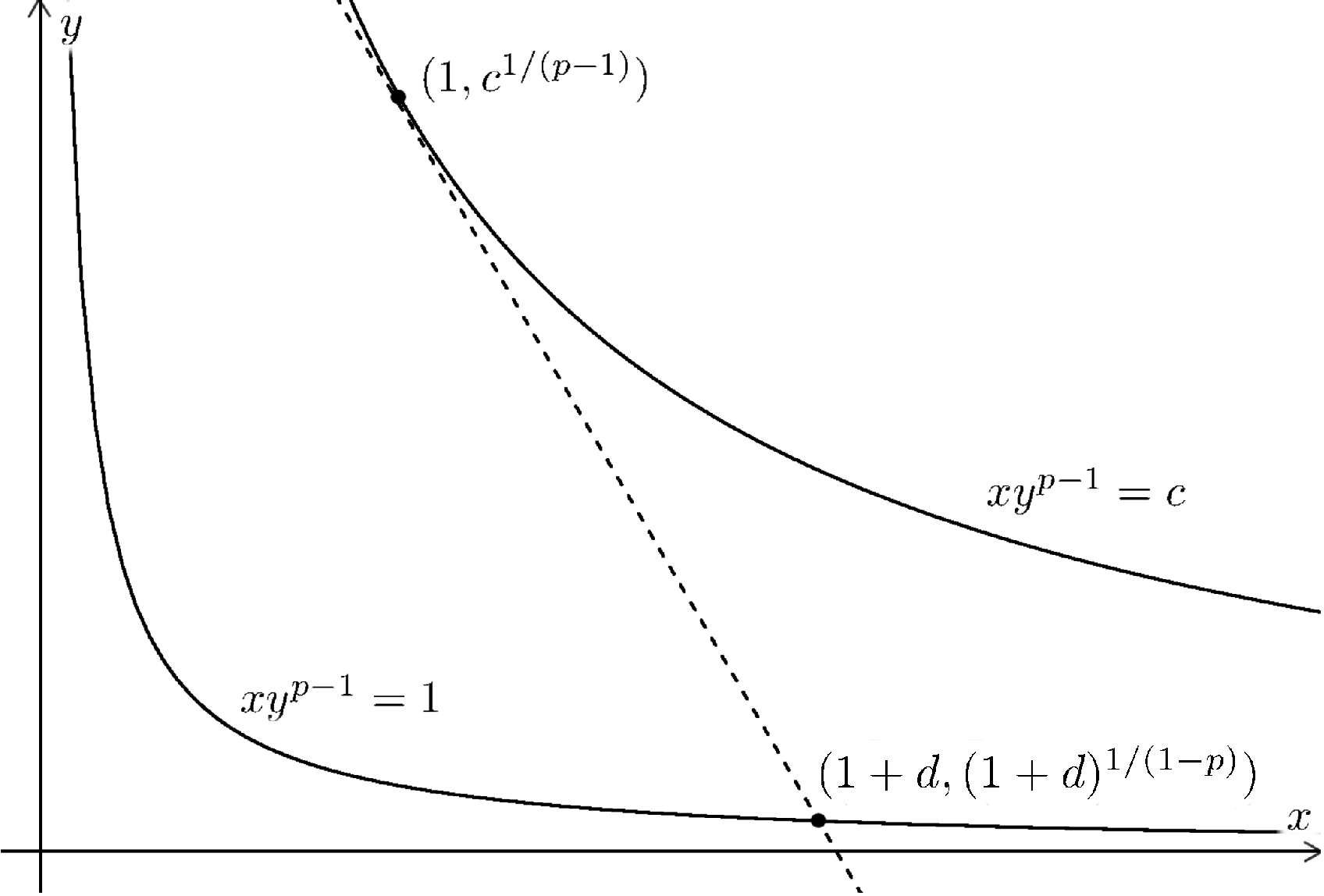}
\caption{The tangent to the curve $xy^{p-1}=c$ at the point $(1,c^{1/(p-1)})$ intersects the curve $xy^{p-1}=1$ at $(1+d,(1+d)^{1/(p-1)})$.}\label{interd}
\end{center}
\end{figure} 
Here is the full weighted extension of \eqref{HLD} (cf. \cite{Os}).

\begin{theorem}
For any $1<p<\infty$ we have the inequality
$$ \|M\|_{L^p(w)\to L^p(w)}\leq \frac{p}{p-1-d(p,[w]_{A_p})}.$$
The constant on the right is the best possible: for any $c\geq 1$ and $\e>0$ there is an $A_p$ weight $w$ satisfying $[w]_{A_p}\leq c$ such that
$$ \|M\|_{L^p(w)\to L^p(w)}\geq \frac{p}{p-1-d(p,c)}-\e.$$
\end{theorem}

To see that this is indeed an extension of \eqref{HLD}, note that the unweighted context corresponds to the choice $[w]_{A_p}=1$; since  $d(p,1)=0$, then the constant in the above $L^p$ bound is equal to $p'$. Our primary goal is to establish the multilinear analogue of the above theorem.

\smallskip

\emph{Multilinear case.} Suppose that $m\geq 1$ is a fixed integer. Then  the $m$-linear dyadic maximal operator $\mathcal{M}$ acts on vectors $\vec{f}=(f_1,f_2,\ldots,f_m)$ of locally integrable functions on $\R^N$ by
$$ \mathcal{M}\vec{f}=\sup_{Q\in \mathscr{D}} \prod_{j=1}^m \langle |f_j|\rangle_Q\chi_Q.$$
A direct combination of \eqref{HLD} and the H\"older inequality implies that if the vector $\vec{P}=(p_1,p_2,\ldots,p_m)$ of exponents satisfies $1<p_1,p_2,\ldots,p_m\leq \infty$ and $\frac{1}{p}=\frac{1}{p_1}+\frac{1}{p_2}+\ldots+\frac{1}{p_m},$ then we have the sharp bound
$$ \|\mathcal{M}\vec{f}\|_{L^p}\leq \prod_{j=1}^m p_j'\|f_j\|_{L^{p_j}}.$$
One can ask about the weighted version of this estimate, motivated by the above discussion for $M$. The precise answer was given by Lerner et. al. in \cite{LOPTT}. 
Given a vector $\vec{w}=(w_1,w_2,\ldots,w_m)$ of weights on $\R^N$, we set
$$ v_{\vec{w}}=\prod_{j=1}^m w_j^{p/p_j}$$
and say that $\vec{w}$ satisfies the dyadic multilinear condition $A_{\vec{P}}$, if
$$ [\vec{w}]_{A_{\vec{P}}}:=\sup_{Q\in \mathscr{D}} \langle v_{\vec{w}}\rangle_Q \prod_{j=1}^m \langle \sigma_j\rangle_Q^{p/p_j'}<\infty.$$
Here, as before, $\sigma_j=w_j^{1-p_j'}$ is the dual weight to $w_j$, $j=1,\,2,\,\ldots,\,m$. 
Note that for $m=1$, the above requirement reduces to the classical Muckenhoupt's condition $A_p$. Actually, the connection to the one-dimensional setting goes deeper: as proved in \cite{LOPTT}, we have $\vec{w}\in A_{\vec{P}}$ if and only if $\sigma_j\in A_{mp_j'}$ for all $j$ and $v_{\vec{w}}\in A_{mp}$. The multilinear $A_{\vec{P}}$ condition provides us with the answer to the above question: $\mathcal{M}$ is bounded as an operator from $L^{p_1}(w_1)\times L^{p_2}(w_2)\times \ldots \times L^{p_m}(w_m)$ to $L^p(v_{\vec{w}})$ if and only if $\vec{w}\in A_{\vec{P}}$ (cf. \cite{LOPTT}). What about the dependence of the norm on the characteristic of $\vec{w}$? In \cite{DLP}, the authors established the following version of \eqref{Buckley}:
\begin{equation}\label{multilinear}
 \|\mathcal{M}\vec{f}\|_{L^p(v_{\vec{w}})}\leq \kappa_{\vec{P}}[\vec{w}]_{A_{\vec{P}}}^{\beta(\vec{P})} \prod_{j=1}^m \|f_j\|_{L^{p_j}(w_j)},
\end{equation}
with $\beta(\vec{P})$ satisfying $ \frac{m}{mp-1}\leq \beta(\vec{P}) \leq \frac{1}{p}\left(1+\sum_{j=1}^m \frac{1}{p_j-1}\right)$. Furthermore, in the special case $p_1=p_2=\ldots=p_m=mp$, they showed that the estimate holds with $\beta(\vec{P})=m/(mp-1)=p_1'/p$, which is optimal. The question about the best exponent for an arbitrary vector $\vec{P}$ was answered in \cite{LKS}: the optimal choice is 
$ \beta(\vec{P})=\max\left\{{p_1'}/{p},{p_2'}/{p},\ldots,{p_m'}/{p}\right\}.$

The purpose of this paper is to present a different approach to the estimate \eqref{multilinear}, which will yield sharp  information about the size of the constant $\kappa_{\vec{P}}$. Actually, we will work in the context of probability spaces equipped with a tree-like structure. Here is the precise definition (cf. \cite{Me}).

\begin{definition}
Suppose that  $(X,\mu)$ is a nonatomic probability space. 
A set $\mathcal{T}$ of measurable subsets of $X$ will be called a tree if the following conditions are satisfied:
\begin{itemize}
\item[(i)] $X\in \mathcal{T}$ and for every $Q\in\mathcal{T}$ we have $\mu(Q)>0$.

\smallskip

\item[(ii)] For every $Q \in \mathcal{T}$ there is a  finite subset $C(Q) \subset \mathcal{T}$ containing at least two elements such that

(a) the elements of $C(Q)$ are pairwise disjoint subsets of $Q$,

(b) $ Q = \bigcup_{Q'\in C(Q)}Q'$.

\smallskip

\item[(iii)] $\mathcal{T} = \bigcup_{n\geq 0} \mathcal{T}^n$, where $\mathcal{T}^0=\{X\}$ and $\mathcal T^{n+1} = \bigcup_{Q\in\mathcal{T}^n} C(Q)$.

\smallskip

\item[(iv)] We have $\lim_{n\to\infty}\sup_{Q\in\mathcal{T}^n} \mu(Q) = 0$. 
\end{itemize}
\end{definition}

An important example, which links this definition with the preceding considerations,  is the cube $X=[0,1)^N$ endowed with Lebesgue measure and the tree of its dyadic subcubes. Any probability space equipped with a tree gives rise to the corresponding multilinear maximal operators $\mathcal{M}=\mathcal{M}_{X,\mathcal{T}}$ and Muckenhoupt's classes $A_{\vec{P}}=A_{\vec{P}}(X,\mathcal{T})$: the definitions are word-by-word the same, one only needs to change the base space $\R^N$ to $X$ and replace the dyadic lattice $\mathscr{D}$ by the tree $\mathcal{T}$.

We turn our attention to the formulation of our main result. For a given $c\geq 1$ and a vector $\vec{P}=(p_1,p_2,\ldots,p_m)\in (1,\infty)^m$ with $p^{-1}=p_1^{-1}+p_2^{-1}+\ldots+p_m^{-1}$, let $\alpha=\alpha(\vec{P},c)\in [0,\min_j p_j-1)$ be a unique solution to the equation
\begin{equation}\label{defalfa}
 c(\alpha+1)\prod_{j=1}^m (1+\alpha(1-p_j'))^{p/p_j'}=1.
\end{equation}

\begin{theorem}\label{mainthm}
Let $(X,\mu)$ be an arbitrary nonatomic probability space with a tree structure $\mathcal{T}$. Then for any vector $\vec{P}=(p_1,p_2,\ldots,p_m)\in (1,\infty)^m$ with $p^{-1}=p_1^{-1}+p_2^{-1}+\ldots+p_m^{-1}$ and any vector $\vec{w}\in A_{\vec{P}}(X,\mathcal{T})$ we have the inequality
\begin{equation}\label{mainin}
\|\mathcal{M}_{X,\mathcal{T}}\|_{L^{p_1}(w_1)\times L^{p_2}(w_2)\times \ldots \times L^{p_m}(w_m)\to L^p(v_{\vec{w}})}\leq  \prod_{j=1}^m \frac{p_j}{p_j-1-\alpha(p,[\vec{w}]_{A_{\vec{P}}})}.
\end{equation}
The constant on the right is the best possible: for any $c\geq 1$ and $\e>0$ there is a weight $\vec{w}$ satisfying $[\vec{w}]_{A_{\vec{P}}}\leq c$ such that
\begin{equation}\label{rever}
 \|\mathcal{M}_{X,\mathcal{T}}\|_{L^{p_1}(w_1)\times L^{p_2}(w_2)\times \ldots \times L^{p_m}(w_m)\to L^p(v_{\vec{w}})}>\prod_{j=1}^m \frac{p_j}{p_j-1-\alpha(p,c)}-\e.
 \end{equation}
\end{theorem}

In particular, in \eqref{mainin} we recover the optimal dependence on the characteristic $[\vec{w}]_{A_{\vec{P}}}$ from \cite{LKS}. To see this, re-order the exponents so that $p_1=p_2=\ldots=p_k<p_{k+1}\leq p_{k+2}\leq \ldots \leq p_m$ and rewrite \eqref{defalfa} in the form
$$ c(\alpha+1)(p_1-1-\alpha)^{kp/p_1'}\prod_{j=k+1}^m (p_j-1-\alpha)^{p/p_j'}=(p_1-1)^{kp/p_1'}\prod_{j=k+1}^m (p_j-1)^{p/p_j'},$$
so that $p_1-1-\alpha=O(c^{-p_1'/kp})$ as $c\to \infty$. Plugging this into \eqref{mainin}, we obtain
 $$ \prod_{j=1}^m \frac{p_j}{p_j-1-\alpha(p,[\vec{w}]_{A_{\vec{P}}})}\leq \kappa_{\vec{P}} [\vec{w}]_{A_{\vec{P}}}^{p_1'/p},$$
 and hence the exponent is indeed $p_1'/p=\max\{p_1'/p,p_2'/p,\ldots,p_m'/p\}$.

  We would also like to emphasize that the estimate \eqref{mainin} is sharp for each individual probability space. Finally, let us mention that standard translation and scaling arguments allow to extend the above results to the non-probabilistic, dyadic context studied at the beginning.

%, which acts on vectors $\vec{f}=(f_1,f_2,\ldots,f_m)$ of integrable functions (random variables) $f_j:X\to \R$ by 
%$$ \mathcal{M}_\mathcal{T}f(x)=\sup_{Q\in \mathcal{T}} \prod_{j=1}^m \langle |f_j|\rangle_{Q,\mu},$$
%where $\langle \varphi\rangle_{Q,\mu}=\frac{1}{\mu(Q)}\int_Q \varphi\mbox{d}\mu$ is the average of $\varphi$ over $Q$ with respect to the %measure $\mu$.
%One defines $A_{\vec{P}}$ weights analogously. Given a vector $\vec{P}=(p_1,p_2,\ldots,p_m)$ of parameters as above, we say that a vector $\vec{w}=(w_1,w_2,\ldots,w_m)$ of positive functions on $X$ satisfies condition $A_{\vec{P}}$, if
%$$ [\vec{w}]_{A_{\vec{P}}}:=\sup_{Q\in \mathcal{T}} \langle v_{\vec{w}}\rangle_Q \prod_{j=1}^m \langle w_j^{1-p_j'}\rangle_Q^{p/p_j'}<\infty.$$

As an application, we obtain the following multilinear Hardy-type inequality.

\begin{theorem}\label{Hardythm}
Let $\vec{P}=(p_1,p_2,\ldots,p_m)\in (1,\infty)^m$ and $p^{-1}=p_1^{-1}+p_2^{-1}+\ldots+p_m^{-1}$. Assume further that  $\vec{w}=(w_1,w_2,\ldots,w_m)$ is a vector of positive, locally integrable functions on $(0,\infty)$. Suppose that the function $v_{\vec{w}}=\prod_{j=1}^m w_j^{p/p_j}$ satisfies
$$ [\vec{w}]_{A_{\vec{P}}}:=\sup \langle v_{\vec{w}}\rangle_I \prod_{j=1}^m \langle w_j^{1-p_j'}\rangle_I^{p/p_j'}<\infty,$$
where the supremum is taken over all intervals  $I=[0,t]$, $t\in \R$. Then for any vector $\vec{f}=(f_1,f_2,\ldots,f_m)$ with $f_j\in L^{p_j}(w_j)$, we have
\begin{equation}\label{Hardy}
\begin{split}
\left(\int_0^\infty \prod_{j=1}^m \left(\int_0^t |f_j(s)|\mbox{d}s\right)^p\frac{v_{\vec{w}}(t)}{t^{mp}}\mbox{d}t\right)^{1/p} & \leq \prod_{j=1}^m \frac{p_j\|f_j\|_{L^{p_j}(w_j)}}{p_j-1-\alpha(p,[\vec{w}]_{A_{\vec{P}}})}.
\end{split}
\end{equation}
The constant $\prod_{j=1}^m p_j/(p_j-1-\alpha(p,[\vec{w}]_{A_{\vec{P}}}))$ is the best possible.
\end{theorem}

Our reasoning will rest on a multilinear version of Carleson embedding theorem and a tight `testing' estimate \eqref{testing} below. To show the latter bound, we will exploit the so-called Bellman function method, a powerful tool for establishing inequalities, used widely in probability and analysis. The Bellman function we invent involves as many as $m+2$ variables, but the calculations are quite quick and do not require elaborate analysis.

The paper is organized as follows. In the next section, we present the proof of the estimate \eqref{mainin}. Section 3 is devoted to the lower bounds for the constants involved. The final part of the paper contains the description of informal steps which lead to the discovery of the best constants and the Bellman function leading to our main result.

\medskip

\section{A special function and its properties}
Throughout this section, $c>1$ is a given number and $\vec{P}=(p_1,p_2,\ldots,p_m)$ is a fixed vector of exponents satisfying $p^{-1}=p_1^{-1}+p_2^{-1}+\ldots+p_m^{-1}$ and $p_j>1$ for each $j$. Furthermore, $\alpha=\alpha(\vec{P},c)$ is defined as in \eqref{defalfa} and we introduce the auxiliary constant
\begin{equation}\label{defCe}
 C=C_{\vec{P},c}=\prod_{j=1}^m \big(1+\alpha(\vec{P},c)(1-p_j')\big)^{-1}.
 \end{equation}
For a fixed sequence $s^*=(s^*_1,s^*_2,\ldots,s^*_m)$ of positive numbers, it will be convenient to use the multiindex-type notation
$$ (s^*)^{p/\underline{p}}=(s_1^*)^{p/p_1}(s_2^*)^{p/p_2}\ldots (s_m^*)^{p/p_m}$$
and
$$ (s^*)^{-p/\underline{p}'}=(s_1^*)^{-p/p_1'}(s_2^*)^{-p/p_2'}\ldots (s_m^*)^{-p/p_m'}.$$
Consider the function $\beta:(0,\infty)^m\to \R$, given by
%$$ B^{s^*}(s;t,u)=pC^p(s^*)^{p/\underline{p}}\left[\frac{\alpha+1}{\alpha p}\left(\frac{u}{(s^*)^{-p/\underline{p}'}}-1\right)+\sum_{j=1}^m\left(\frac{\alpha+1}{\alpha p_j'}-1\right)\left(\frac{s_j}{s_j^*}-1\right)\right]$$
$$ \beta(s_1,s_2,\ldots,s_m)=pC^p\left[\sum_{j=1}^m\left(\frac{\alpha+1}{\alpha p_j'}-1\right)s_j-\frac{m}{\alpha}\right].$$
The central Bellman function of this paper is the function $B:(0,\infty)^m\times (0,\infty)\times (0,\infty)\to \R$ defined by the formula
$$ B(s;t,u)=\frac{\alpha+1}{\alpha} t^pu+\inf\left\{(s^*)^{p/\underline{p}}\beta\left(\frac{s_1}{s_1^*},\frac{s_2}{s_2^*},\ldots, \frac{s_m}{s_m^*}\right)\right\},$$
where the infimum is taken over all sequences $s^*=(s_1^*,s_2^*,\ldots,s_m^*)$ of positive numbers satisfying $Cs_1^*s_2^*\ldots s_m^*=t$.
For a fixed $t$, the function $B$, treated as the function of $u$ and $s$, is the infimum of affine functions, so it is concave. We will prove the following two key properties of this object.

\begin{lemma}\label{propB1}
For any $s=(s_1,s_2,\ldots,s_m)\in (0,\infty)^m$ and any $u\geq s^{-p/\underline{p}'}$, we have
\begin{equation}\label{maj}
 B(s;t,u)\geq t^ps^{-p/\underline{p}'}-C^ps^{p/\underline{p}}.
\end{equation}
\end{lemma}
\begin{proof}
Clearly, we have $\partial_u B\geq 0$ (see the formula for $B$). Therefore, all we need is the estimate
$$  B(s;t,s^{-p/\underline{p}'})\geq t^ps^{-p/\underline{p}'}-C^ps^{p/\underline{p}}.
$$
To this end, by the very definition of $B$, we must show that
$$ \frac{t^ps^{-p/\underline{p}'}}{\alpha}+(s^*)^{p/\underline{p}}\beta\left(\frac{s_1}{s_1^*},\frac{s_2}{s_2^*},\ldots,\frac{s_m}{s_m^*}\right)+C^p s^{p/\underline{p}}\geq 0$$
for any $s_1^*,\,s_2^*,\,\ldots,\,s_m^*>0$ with $C s_1^*s_2^*\ldots s_m^*=t$. Let us divide throughout by $mpC^p(s^*)^{p/\underline{p}}/\alpha$ and substitute $r_j=s_j/s_j^*$ to obtain the equivalent form
$$ \frac{r^{-p/\underline{p}'}}{pm}+\sum_{j=1}^m \frac{\alpha+1-\alpha p_j'}{mp_j'} r_j+\frac{\alpha}{pm}r^{p/\underline{p}}\geq 1.$$
This bound follows immediately from Young's inequality, because of the two identities
$$ \frac{1}{pm}+\sum_{j=1}^m \frac{\alpha+1-\alpha p_j'}{mp_j'} +\frac{\alpha}{pm}=1$$
and
$$ \big(r^{-p/\underline{p}'}\big)^{1/pm}\cdot \prod_{j=1}^m r_j^{(\alpha+1-\alpha p_j')/mp_j'} \cdot \big(r^{p/\underline{p}}\big)^{\alpha/pm}=1.$$
The proof is complete.
\end{proof}

\begin{lemma}\label{propB2}
Suppose that $(s;t,u)\in (0,\infty)^m\times (0,\infty)\times (0,\infty)$ satisfies $ s_1s_2\ldots s_m\geq t$ and $ u\leq cs^{-p/\underline{p}'}$. Then
\begin{equation}\label{monoto}
 B(s;t,u)\geq B(s;s_1s_2\ldots s_m,u).
\end{equation}
\end{lemma}
\begin{proof}
By the very definition of $B$, we have
\begin{equation}\label{partial}
\partial_u B(s;t,u)=(\alpha+1)t^p/\alpha\leq (\alpha+1)(s_1s_2\ldots s_m)^p/\alpha=\partial_u B(s;s_1s_2\ldots s_m,u),
\end{equation}
and hence it is enough to establish \eqref{monoto} for the largest value of $u$, i.e., for $u=cs^{-p/\underline{p}'}$. First we handle the right-hand side: we have
\begin{equation}\label{zero}
 B(s;s_1s_2\ldots s_m,cs^{-p/\underline{p}'})\leq 0.
\end{equation}
To see this, we pick $s^*_j=s_j(1+\alpha(1-p_j'))$, $j=1,\,\,2,\,\ldots,\,m$. Then
$$ Cs_1^*s_2^*\ldots s_m^*=Cs_1s_2\ldots s_m\cdot \prod_{j=1}^m(1+\alpha(1-p_j'))=s_1s_2\ldots s_m,$$
so by the very definition of $B$,
$$ B(s;s_1s_2\ldots s_m,cs^{-p/\underline{p}'})\leq \frac{\alpha+1}{\alpha}(s_1s_2\ldots s_m)^p\cdot cs^{-p/\underline{p}'} +(s^*)^{p/\underline{p}}\beta\left(\frac{s_1}{s_1^*},\frac{s_2}{s_2^*},\ldots, \frac{s_m}{s_m^*}\right).$$
However, by \eqref{defalfa} and \eqref{defCe} we compute that
$$ (s^*)^{p/\underline{p}}=s^{p/\underline{p}}\prod_{j=1}^m \frac{(1+\alpha(1-p_j'))^p}{(1+\alpha(1-p_j'))^{p/p_j'}}=s^{p/\underline{p}}\cdot \frac{c(\alpha+1)}{C^p}$$
and
\begin{align*}
 \beta\left(\frac{s_1}{s_1^*},\frac{s_2}{s_2^*},\ldots, \frac{s_m}{s_m^*}\right)&=pC^p\left[\sum_{j=1}^m \left(\frac{\alpha+1}{\alpha p_j'}-1\right)\frac{1}{1+\alpha(1-p_j')}-\frac{m}{\alpha}\right]\\
 &=\frac{pC^p}{\alpha}\left(\sum_{j=1}^m \frac{1}{p_j'}-m\right)=-\frac{C^p}{\alpha}.
\end{align*}
Combining the above observations, we obtain \eqref{zero}; thus, to complete the proof, it is enough to prove that
$ B(s;t,cs^{-p/\underline{p}'})\geq 0.$ By the definition of $B$, we need to show the estimate
$$ \frac{\alpha+1}{\alpha}t^p\cdot cs^{-p/\underline{p}'}+(s^*)^{p/\underline{p}}\beta\left(\frac{s_1}{s_1^*},\frac{s_2}{s_2^*},\ldots, \frac{s_m}{s_m^*}\right)\geq 0$$
for any sequence $s_1^*$, $s_2^*$, $\ldots$, $s_m^*$ of positive numbers satisfying $Cs_1^*s_2^*\ldots s_m^*=t$. Recall that by \eqref{defalfa},
$$ c(\alpha+1)=\prod_{j=1}^m (1+\alpha(1-p_j'))^{-p/p_j'}.$$
The substitution $r_j=(1+\alpha(1-p_j'))s_j/s_j^*$ transforms the desired bound into
$$ \frac{r^{-p/\underline{p}'}}{pm}+\sum_{j=1}^m \frac{r_j}{mp_j'}\geq 1,$$
which follows directly from Young's inequality.
\end{proof}

\begin{remark}
The inequality \eqref{zero} and the bound $\partial_u B\geq 0$ (see the formula for $B$) imply that for any $s_1,s_2,\ldots,s_m>0$ and any $ s^{-p/\underline{p}'}\leq u\leq cs^{-p/\underline{p}'}$ we have
\begin{equation}\label{init}
B(s_1,s_2,\ldots,s_m;s_1s_2\ldots s_m,u)\leq 0.
\end{equation}
\end{remark}

\medskip

\section{Proof of \eqref{mainin} and \eqref{Hardy}}
Throughout, $\vec{P}=(p_1,p_2,\ldots,p_m)$ is a vector of exponents belonging to $(1,\infty)$, the parameter $p$ is given by $\frac{1}p=\frac{1}{p_1}+\frac{1}{p_2}+\ldots+\frac{1}{p_m}$ and the constant $C_{\vec{P},c}$ comes from \eqref{defCe}. The Bellman function $B$ introduced above leads to the following result.
\begin{theorem}
Suppose that $(X,\mu)$ is a probability space with a tree structure $\mathcal{T}$.
Fix $c\geq 1$ and let $\vec{w}$ be a vector of weights on $X$, satisfying $[\vec{w}]_{A_{\vec{P}}}\leq c$. Then for any $R\in \mathcal{T}$, the vector $\vec{\sigma}=(\sigma_1,\sigma_2,\ldots,\sigma_m)$ of dual weights satisfies
\begin{equation}\label{testing}
 \int_R \mathcal{M}(\vec{\sigma}\chi_R)^p v_{\vec{w}}\mbox{d}\mu\leq C^p_{\vec{P},c}\int_R \prod_{j=1}^m \sigma_j^{p/p_j}\mbox{d}\mu .
\end{equation}
\end{theorem}
\begin{proof}
It is convenient to split the reasoning into a few intermediate parts.

\smallskip

\emph{Step 1. An auxiliary notation.}
Fix $w$ and $R$ as in the statement and let $n_0$ be the unique integer such that $R\in \mathcal{T}^{n_0}$. We introduce the auxiliary functional sequences $(x_n)_{n\geq n_0}$, $(y_n)_{n\geq n_0}$ and $(z_n)_{n\geq n_0}$ as follows: for any $n\geq n_0$ and $\omega\in R$,
\begin{align*}
 x_n(\omega)&=(x_{n1}(\omega),x_{n2}(\omega),\ldots,x_{nm}(\omega))=\big(\langle \sigma_1\rangle_{Q_n(\omega)},\langle \sigma_2\rangle_{Q_n(\omega)},\ldots,\langle \sigma_m\rangle_{Q_n(\omega)}\big)\\
 y_n(\omega)&=\max_{n_0\leq k\leq n}\Big(x_{k1}(\omega)x_{k2}(\omega)\ldots x_{km}(\omega)\Big),\\
z_n(\omega)&=\langle v_{\vec{w}}\rangle_{Q_n(\omega)},
\end{align*}
where $Q_n(\omega)$ is the unique element of $\mathcal{T}^n$ which contains $\omega$. There is a nice probabilistic interpretation of these sequences: $(x_n)_{n\geq n_0}$ and $(z_n)_{n\geq n_0}$ are martingales induced by the filtration $(\mathcal{T}^n)_{n\geq n_0}$ , with the terminal variables equal to $(\sigma_1,\sigma_2,\ldots,\sigma_m)$ and $v_{\vec{w}}$, respectively; in addition, $(y_n)_{n\geq n_0}$ is the maximal process associated with the sequence of products $(x_{k1}x_{k2}\ldots x_{km})_{k\geq n_0}$. Observe that by Lebesgue's differentiation theorem (or rather, by Doob's martingale convergence theorem), we have $x_{nj}\to \sigma_j$, $y_n\to \mathcal{M}(\vec{\sigma}\chi_R)$ and $z_n\to v_{\vec{w}}$ almost surely as $n\to \infty$.

\smallskip

\emph{Step 2. Monotonicity.} Now we will show that the sequences $(x_n)_{n\geq n_0}$, $(y_n)_{n\geq n_0}$ and $(z_n)_{n\geq n_0}$ combine nicely with the Bellman function $B$ defined previously. More specifically, we will prove that the sequence
$$ \left(\int_R B(x_n;y_n,z_n)\mbox{d}\mu\right)_{n\geq n_0}$$
is nonincreasing. To see this, fix an integer $n\geq n_0$, let $Q$ be an arbitrary element of $\mathcal{T}^n$ and let $Q_1$, $Q_2$, $\ldots$, $Q_\ell$ be the collection of all children of $Q$ in $\mathcal{T}^{n+1}$. The functions $x_n$, $y_n$, $z_n$ are constant on $Q$, while $x_{n+1}$, $y_{n+1}$ and $z_{n+1}$  are constant on each $Q_j$. It is easy to check that these constant values satisfy
\begin{equation}\label{constan}
 \langle x_n\rangle_Q=\sum_{j=1}^\ell \frac{|Q_j|}{|Q|} \langle x_{n+1}\rangle_{Q_j},\qquad \langle z_n\rangle_Q =\sum_{j=1}^\ell \frac{|Q_j|}{|Q|}\langle z_{n+1}\rangle_{Q_{j}}.
\end{equation}
(This is nothing but the martingale property, expressed in analytic terms).  Next, observe that for any $j$ we have
\begin{equation}\label{interm}
 \int_{Q_j} B(x_{n+1};y_{n+1},z_{n+1})\mbox{d}\mu\leq \int_{Q_j} B(x_{n+1};y_n,z_{n+1})\mbox{d}\mu.
\end{equation}
Indeed, if $y_{n+1}=y_n$ on $Q_j$, then there is nothing to prove; on the other hand, if $y_{n+1}\neq y_n$ on $Q_j$, then by the definition of the sequence $y$ we must have $y_n<y_{n+1}=x_{(n+1)1}x_{(n+1)2}\ldots x_{(n+1)m}$. Furthermore, by the $A_p$ condition we have
$$ z_{n+1}(\omega)=\langle v_{\vec{w}}\rangle_{Q_{n+1}(\omega)}\leq c\prod_{j=1}^m \langle \sigma_j\rangle_{Q_{n+1}(\omega)}^{-p/p_j'}=c x_{(n+1)1}^{-p/p_1'}(\omega)x_{(n+1)2}^{-p/p_2'}(\omega)\ldots x_{(n+1)m}^{-p/p_m'}(\omega). $$
By Lemma \ref{propB2}, this implies $B(x_{n+1};y_{n+1},z_{n+1})\leq  B(x_{n+1};y_n,z_{n+1})$ on $Q_j$ and hence \eqref{interm} follows. Summing over $j$, we thus obtain
$$  \int_{Q} B(x_{n+1};y_{n+1},z_{n+1})\mbox{d}\mu\leq \int_{Q} B(x_{n+1};y_n,z_{n+1})\mbox{d}\mu.$$
Now, for a fixed $t$, the function $(s,u)\to B(s;t,u)$, being the infimum of affine functions, is concave. Therefore, by \eqref{constan}, the right-hand side above does not exceed  $ \int_{Q} B(x_{n};y_n,z_{n})\mbox{d}\mu.$
Summing over all $Q\in \mathcal{T}^n$ contained in $R$, we obtain the aforementioned monotonicity.

\smallskip

\emph{Step 3. Completion of the proof.} By the previous step, we obtain that for any $n\geq n_0$ we have
\begin{equation}\label{land}
 \int_R B(x_{n};y_n,z_n)\mbox{d}\mu\leq \int_R B(x_{n_0};y_{n_0},z_{n_0})\mbox{d}\mu.
\end{equation}
Let us inspect the expression on the right. By the very definition of $y$, we have $y_{n_0}=x_{n_01}x_{n_02}\ldots x_{n_0m}$; furthermore,  by the $A_{\vec{P}}$ condition, we obtain $z_{n_0}\leq c\prod_{j=1}^m x_{n_0j}^{-p/p_j'}$. Consequently, we have $
 B(x_{n_0};y_{n_0},z_{n_0})\leq 0$, by virtue of \eqref{init}.
The left-hand side of \eqref{land} is handled by Lemma \ref{propB1}: we have
$$ B(x_{n};y_n,z_n)\geq y_n^px_{n1}^{-p/p_1'}x_{n2}^{-p/p_2'}\ldots x_{nm}^{-p/p_m'}-C^px_{n1}^{p/p_1}x_{n2}^{p/p_2}\ldots x_{nm}^{p/p_m}.$$Putting all these fact together, we get
$$ \int_R y_n^px_{n1}^{-p/p_1'}x_{n2}^{-p/p_2'}\ldots x_{nm}^{-p/p_m'}\mbox{d}\mu\leq C^p \int_R \prod_{j=1}^m x_{nj}^{p/p_j}\mbox{d}\mu.$$
It remains to let $n\to \infty$ and carry out an appropriate limiting procedure. Recall that we have the almost sure convergence $x_{nj}\to \sigma_j$ and $z_n\to v_{\vec{w}}$. Therefore, the left-hand side above can be handled by Fatou's lemma:
$$ \int_R \mathcal{M}(\vec{\sigma}\chi_R)^p v_{\vec{w}}\mbox{d}\mu\leq \liminf_{n\to\infty} \int_R y_n^px_{n1}^{-p/p_1'}x_{n2}^{-p/p_2'}\ldots x_{nm}^{-p/p_m'}\mbox{d}\mu.$$
To deal with the right hand side, we will apply Lebesgue's dominated convergence theorem: we have $\prod_{j=1}^m x_{nj}^{p/p_j}\to \prod_{j=1}^m \sigma_j^{p/p_j}$ almost surely, so all we need is a suitable majorant. The inclusion $\vec{w}\in A_{\vec{P}}$ implies $\sigma_j\in A_{mp_j'}$ (cf. \cite{LOPTT}), which gives that $\sigma_j$, and hence also the maximal function $M\sigma_j$, belongs to $L^r$ for some $r>1$. Thus, by the H\"older inequality, the majorant $\prod_{j=1}^m \Big(\sup_{n\geq n_0}x_{nj}\Big)^{p/p_j}$ is integrable and the assertion follows.
\end{proof}

Now, recall the following multilinear version of Carleson embedding theorem established in \cite{CD} (see also \cite{HP}). Here and below, we use the notation $\langle \psi \rangle_{Q,\sigma_j}$ for the average of $\psi$ over $Q$ with respect to the measure $\sigma_j\mbox{d}\mu$.

\begin{theorem}\label{multilinear-carleson}
Suppose that a sequence $(a_Q)_{Q\in \mathcal{T}}$ of nonnegative numbers satisfies the following condition: for any $R\in \mathcal{T}$ we have the estimate
\begin{equation*}\label{Carle}
 \sum_{Q\subseteq R} a_Q \leq K^p\int_R \prod_{j=1}^m \sigma_j^{p/p_j}\mbox{d}\mu.
\end{equation*}
Then for any vector $\vec{f}=(f_1,f_2,\ldots,f_m)$ such that $f_j\in L^{p_j}(\sigma_j)$ for all $j$, we have
\begin{equation*}\label{innu}
 \left(\sum_{Q\in \mathcal{T}} a_Q \prod_{j=1}^m \langle f_j\rangle_{Q,\sigma_j}^p\right)^{1/p}\leq K\prod_{j=1}^m p_j'\|f_j\|_{L^{p_j}(\sigma_j)},
\end{equation*}
where the notation $\langle \psi \rangle_{Q,\sigma_j}$ denotes the average of $\psi$ over $Q$ with respect to $\sigma_j\mbox{d}\mu$.
\end{theorem}

We are ready for the proof of our main estimate stated in Theorem \ref{mainthm}.

\begin{proof}[Proof of \eqref{mainin}] Fix a vector $\vec{w}\in A_{\vec{P}}(X,\mathcal{T})$ and pick any sequence $\vec{f}=(f_1,f_2,\ldots,f_m)$ of functions on $X$ such that $f_j\in L^{p_j}(w_j)$ for each $j$. Clearly, we may assume that $f_j>0$ for all $j$. Let $\e>0$ be a fixed parameter. %By a straightforward approximation, we may assume that there is an integer $L$ such that $\vec{f}$ is measurable with respect to the $\sigma$-algebra generated by $\mathcal{T}^L$ (that is, $\vec{f}$ is constant on each atom of $\mathcal{T}^L$). Under this simplicity assumption, we see that the suprema defining the multilinear maximal functions are attained:
By the definition of the multilinear maximal operator, for any $\omega\in X$ there is a set $\mathcal{Q}(\omega)$ containing $\omega$ such that
$$ \mathcal{M}\vec{f}(\omega)\leq (1+\e)\langle f_1\rangle_{\mathcal{Q}(\omega)}\langle f_2\rangle_{\mathcal Q(\omega)}\ldots \langle f_m\rangle_{\mathcal Q(\omega)}.$$
Of course, such a set $\mathcal Q(\omega)$ need not be unique: to avoid ambiguity, we take $\mathcal{Q}(\omega)$ belonging to $\mathcal{T}^n$ with $n$ as small as possible. Now, for any $Q\in \mathcal{T}$, we define $ E(Q)=\{\omega\in X\,:\,\mathcal{Q}(\omega)=Q\}$
and let
$$a_Q=\prod_{j=1}^m \langle \sigma_j\rangle_Q^p\cdot v_{\vec{w}}(E(Q)).$$
By the very definition, we have $E(Q)\subseteq Q$ and the sets $E(Q)$ corresponding to different $Q$'s are disjoint. Furthermore, we have $a_Q\leq (\mathcal{M}(\vec{\sigma}\chi_Q)(\omega))^pv_{\vec{w}}(E(Q))$ for $\omega\in Q$.
Consequently, for any $R\in \mathcal{T}$ we get
$$ \sum_{Q\subseteq  R}a_Q\leq \int_R (\mathcal{M}(\vec{\sigma}\chi_Q))^pv_{\vec{w}}\mbox{d}\mu\leq C^p\int_R \prod_{j=1}^m \sigma_j^{p/p_j}\mbox{d}\mu,$$
where the second inequality is due to \eqref{testing}.  Thus, by the Carleson embedding theorem (i.e. Theorem \ref{multilinear-carleson}) applied to the vector $\overrightarrow{f\sigma^{-1}}=(f_1\sigma_1^{-1},f_2\sigma_2^{-1},\ldots,f_m\sigma_m^{-1})$, we obtain
$$ \left(\sum_{Q\in \mathcal{T}} a_Q \prod_{j=1}^m \langle f_j\sigma_j^{-1}\rangle_{Q,\sigma_j}^p\right)^{1/p}\leq C\prod_{j=1}^m p_j'\|f_j\sigma_j^{-1}\|_{L^{p_j}(\sigma_j)}=C\prod_{j=1}^m p_j'\|f_j\|_{L^{p_j}(w_j)}.$$
Now, by the definition of $a_Q$ and the identity $\langle f_j\sigma_j^{-1}\rangle_{Q,\sigma_j}=\langle f_j\rangle_Q\langle \sigma_j\rangle_Q^{-1}$, we get
\begin{align*}
\sum_{Q\in \mathcal{T}} a_Q \prod_{j=1}^m \langle f_j\sigma_j^{-1}\rangle_{Q,\sigma_j}^p&=\sum_{Q\in \mathcal{T}}v_{\vec{w}}(E(Q))\prod_{j=1}^m \langle f_j\rangle_Q^p\geq (1+\e)^{-1}\int_X (\mathcal{M}\vec{f})^pv_{\vec{w}}\mbox{d}\mu.
\end{align*}
Combining the two estimates above yields the claim, since $\e$ was arbitrary.
\end{proof}

We now prove the multilinear Hardy inequality formulated in Theorem \ref{Hardythm}.

\begin{proof}[Proof of \eqref{Hardy}] It is enough to show the localized estimate
\begin{equation}\label{local}
 \left(\int_0^T \prod_{j=1}^m \left(\int_0^t |f_j(s)|\mbox{d}s\right)^p\frac{v_{\vec{w}}(t)}{t^{mp}}\mbox{d}t \right)^{1/p} \leq \prod_{j=1}^m \frac{p_j\|f_j\|_{L^{p_j}([0,T],w_j)}}{p_j-1-\alpha(p,[\vec{w}]_{A_{\vec{P}}})}
 \end{equation}
for any $T>0$, and we may restrict ourselves to nonnegative functions $f_j$. This inequality will follow from the estimate \eqref{mainin}, thanks to a discretization argument. Fix $\vec{w}$, $\vec{f}$ as in the statement, let $\delta\in (0,1)$ be a small number and distinguish the points $t_n=(1-\delta)^nT$, $n=0,\,1,\,2,\,\ldots$. Consider the space $X=(0,T]$ equipped with the normalized Lebesgue measure $\mu=|\cdot|/T$. Introduce the tree $\mathcal{T}$ by the following inductive procedure:

\smallskip

(i) We let $\mathcal{T}^0=\{(0,T]\}$;

(ii) If $\mathcal{T}^j=\{(0,t_j],I_1,I_2,\ldots,I_k\}$ for some subintervals $I_1,\,I_2,\,\ldots,\, I_k$, then
$$ \mathcal{T}^{j+1}=\{(0,t_{j+1}],(t_{j+1},t_j],I_{1-},I_{1+},I_{2-},I_{2+},\ldots,I_{k-},I_{k+}\}.$$

\smallskip

\noindent Here $I_-$ and $I_+$ denote the left and the right half of $I$, respectively.
Next, let $\tilde{\sigma}_j,\,\tilde{f_j}:X\to (0,\infty)$ be the conditional expectations of $\sigma_j=w_j^{1-p_j'}$ and $f_j$ with respect to the $\sigma$-algebra generated by $(t_1,t_0]$, $(t_2,t_1]$, $\ldots$, i.e.,
\begin{equation}\label{condexp}
 \tilde{\sigma}_j=\sum_{n=1}^\infty\langle \sigma_j\rangle_{(t_n,t_{n-1}]}\chi_{(t_n,t_{n-1}]},\quad \quad \tilde{f}_{j}=\sum_{n=1}^\infty\langle f_j\rangle_{(t_{n},t_{n-1}]}\chi_{(t_{n},t_{n-1}]}.
\end{equation}
Finally, set $\tilde{w}_j=\tilde \sigma_j^{1-p_j}$. Then the vector weight $\vec{\tilde{w}}=(\tilde{w}_1,\tilde{w}_2,\ldots,\tilde{w}_m)$ satisfies the $A_{\vec{P}}$ condition with respect to the above tree $\mathcal{T}$ and $[\vec{\tilde{w}}]_{A_{\vec{P}}}\leq [\vec{w}]_{A_{\vec{P}}}$. To see this, fix an arbitrary element $Q$ of $\mathcal{T}$. There are two possibilities: either we have $Q=(0,t_n]$ for some $n$, or $Q$ is entirely contained in some interval $(t_{n-1},t_n]$. If the second possibility occurs, then $\vec{\tilde{w}}$ is constant on $Q$ and hence
$$ \langle v_{\vec{\tilde{w}}}\rangle_Q\prod_{j=1}^m \langle \tilde{w}^{1-p_j'}_j\rangle_Q^{p/p_j'}=1.$$
In the case $Q=(0,t_n]$ for some $n$, we note that for any $k\geq 1$ we have
$$  \langle v_{\vec{\tilde{w}}}\rangle_{(t_{k},t_{k-1}]}=\prod_{j=1}^m \langle \tilde{w}^{1-p_j'}_j\rangle_{(t_{k},t_{k-1}]}^{-p/p_j'}=\prod_{j=1}^m \langle \tilde{\sigma}_j\rangle_{(t_{k},t_{k-1}]}^{-p/p_j'}=\prod_{j=1}^m \langle {\sigma}_j\rangle_{(t_{k},t_{k-1}]}^{-p/p_j'}\leq \langle v_{\vec{w}}\rangle_{(t_{k},t_{k-1}]},$$
where the last estimate follows from H\"older's inequality. Summing over all $k\geq n+1$, we obtain $\langle v_{\vec{\tilde{w}}}\rangle_Q\leq \langle v_{\vec{w}}\rangle_Q$ and hence
$$ \langle v_{\vec{\tilde{w}}}\rangle_Q\prod_{j=1}^m \langle \tilde{w}^{1-p_j'}_j\rangle_Q^{p/p_j'}\leq
 \langle v_{\vec{w}}\rangle_Q\prod_{j=1}^m \langle \tilde{\sigma}_j\rangle_Q^{p/p_j'}= \langle v_{\vec{w}}\rangle_Q\prod_{j=1}^m \langle {\sigma}_j\rangle_Q^{p/p_j'}\leq [\vec{w}]_{A_{\vec{P}}}.$$
This establishes the $A_{\vec{P}}$ property and thus we may apply the estimate \eqref{mainin} to $\vec{\tilde{f}}$ and $\vec{\tilde{w}}$.
Note that by the very definition of the multilinear maximal operator, if $t\in (t_{n},t_{n-1}]$, then automatically $t\in (0,t_{n-1}]\in \mathcal{T}$ and hence
$$ \mathcal{M}\vec{\tilde{f}}(t)\geq \prod_{j=1}^m \left(\frac{1}{\mu((0,t_{n-1}])}\int_{(0,t_{n-1}]} \tilde{f_j}\mbox{d}\mu\right)=\frac{1}{t_{n-1}^m}\prod_{j=1}^m \left(\int_0^{t_{n-1}} f_j(s)\mbox{d}s\right).$$
Furthermore, by the convexity of the function $(a,b)\mapsto a^{p_j}b^{1-p_j}$ on $[0,\infty)^2$, we get
$$ \|\tilde{f}_j\|_{L^{p_j}(\tilde{w}_j)}^{p_j}=\int_X \tilde{f}_j^{p_j}\tilde{\sigma}_j^{1-p_j}\mbox{d}\mu\leq \int_X f_j\sigma_j^{1-p_j}\mbox{d}\mu=\frac{1}{T}\|f_j\|_{L^{p_j}([0,T],w_j)}^{p_j}.$$
Therefore, the estimate \eqref{mainin} yields
\begin{align*}
 \left\{\sum_{n=1}^\infty \left[\frac{1}{t_{n-1}^m}\prod_{j=1}^m \left(\int_0^{t_{n-1}}\!\! f_j\right)\right]^p \!\!\!\langle v_{\vec{\tilde{w}}}\rangle_{(t_{n},t_{n-1}]} (t_{n-1}-t_{n})\right\}^{1/p}\!\!\!\!\leq
 \prod_{j=1}^m \frac{p_j\|f_j\|_{L^{p_j}([0,T],\tilde{w}_j)}}{p_j-1-\alpha(p,[\vec{w}]_{A_{\vec{P}}})}
 \end{align*}
 and letting $\delta\to 0$ gives \eqref{local}, by Fatou's lemma.
\end{proof}
\medskip

\section{On the lower bound for the constant}
We have decided to split the contents of this section into two separate parts. In the first part, we construct an explicit example which shows that the constant in the Hardy inequality \eqref{Hardy} is optimal. The calculations for this special case serve as the motivation for the extremal functions for \eqref{mainin}. This general case will be studied in the second part.

\subsection{Sharpness of \eqref{Hardy}} It is instructive to start with the classical Hardy bound
$$ \int_0^\infty \left(\frac{1}{t}\int_0^t |f|\right)^p\mbox{d}t\leq \left(\frac{p}{p-1}\right)^p\int_0^\infty |f|^p,\qquad 1<p<\infty.$$
The sharpness of this estimate is obtained by testing over truncated power functions: for a given $\e\in (0,1)$, let
$$ f^\e(t)=\frac{p}{p-1}\e^{-1/p}\chi_{(0,\e)}(t)+t^{-1/p}\chi_{(\e,\e^{-1})}(t).$$
Then we compute directly that
\begin{align*}
 \int_0^\infty \left(\frac{1}{t}\int_0^t |f^\e|\right)^p\mbox{d}t&\geq \int_\e^{\e^{-1}} \left(\frac{1}{t}\int_0^t |f^\e|\right)^p\mbox{d}t\\
&=\left(\frac{p}{p-1}\right)^p\int_\e^{\e^{-1}} t^{-1}\mbox{d}t=-2\left(\frac{p}{p-1}\right)^p\log \e.
\end{align*}
On the other hand,
$$ \int_0^\infty |f^\e|^p=\left(\frac{p}{p-1}\right)^p-2\log \e,$$
and hence letting $\e\to 0$ shows that the constant $(p/(p-1))^p$ is optimal. Importantly, that the limiting extremal function $\lim_{\e\downarrow 0}f^\e$ does not lie in $L^p$: the integrals on the left and on the right above diverge (with the same, logarithmic rate).

Motivated by \eqref{Hardy}, it is natural to study the sharpness of \eqref{Hardy} with power functions. Fix $\vec{P}$ and $c\in [1,\infty)$, and consider the weights
\begin{equation}\label{weiHardy}
w_1(t)=w_2(t)=\ldots=w_m(t)=t^\alpha,
\end{equation}
where $\alpha=\alpha(\vec{P},c)$ is given by \eqref{defalfa}. Then $v_{\vec{w}}=w_1(t)$ and we check directly that for any $t>0$,
$$ \langle v_{\vec{w}}\rangle_{[0,t]}\prod_{j=1}^m \langle w_j^{1-p_j'}\rangle_{[0,t]}^{p/p_j'}=\left(1+\alpha\right)^{-1}\prod_{j=1}^m \left(1+\alpha(1-p_j')\right)^{-p/p_j'}=c,$$
where the last equality is precisely \eqref{defalfa}. Therefore, $\vec{w}$ is an $A_{\vec{P}}$ weight and we have $[\vec{w}]_{A_{\vec{P}}}=c$.
To introduce $\vec{f}$, we proceed as in the above classical Hardy inequality and pick $f_j$ to be a truncated power function, for which the norm $\|f_j\|_{L^{p_j}(w_j)}$ explodes in the limit. Precisely, for any $\e\in (0,1)$, we take
\begin{equation}\label{funHardy}
 f_j^\e(t)=\frac{p_j}{p_j-1-\alpha}\e^{-(\alpha+1)/p_j}\chi_{(0,\e)}(t)+  t^{-(\alpha+1)/p_j}\chi_{(\e,\e^{-1})}(t).
\end{equation}
The choice of the somewhat complicated expression on $(0,\e)$ is motivated by the following property: for $t\in (\e,\e^{-1})$ we have the simple formula
$$ \frac{1}{t}\int_0^t f^\e_j=\frac{p_jt^{-(\alpha+1)/p_j}}{p_j-1-\alpha}.$$
 Now we proceed as in the above classical case, deriving that
 \begin{align*}
 \int_0^\infty \prod_{j=1}^m \left(\frac{1}{t}\int_0^t f_j^\e\right)^pv_{\vec{w}}(t)\mbox{d}t&\geq \int_\e^{\e^{-1}} \prod_{j=1}^m \left(\frac{1}{t}\int_0^t f_j^\e\right)^pv_{\vec{w}}(t)\mbox{d}t\\
 &=\prod_{j=1}^m \left(\frac{p_j}{p_j-1-\alpha}\right)^p\int_{\e}^{\e^{-1}}t^{-1}\mbox{d}t\\
 &=-2\prod_{j=1}^m \left(\frac{p_j}{p_j-1-\alpha}\right)^p\log \e
 \end{align*}
 and
 $$ \prod_{j=1}^m\|f_j^\e\|_{L^{p_j}(w_j)}=\prod_{j=1}^m\left[\frac{1}{1+\alpha}\left(\frac{p_j}{p_j-1-\alpha}\right)^{p_j}-2\log \e\right]^{1/p_j}.$$
 Letting $\e\to 0$, we see that the best constant in \eqref{Hardy}, under the restriction $[\vec{w}]_{A_{\vec{P}}}=c$, is at least $\prod_{j=1}^m \left(\frac{p_j}{p_j-1-\alpha}\right)$. This is precisely the desired sharpness.

\smallskip

Of course, the above example immediately shows that the constant in the inequality \eqref{mainin} cannot be improved in general. However, the key fact, to be proved in the subsection below, is that this sharpness holds for an \emph{arbitrary} probability space. %Here is the idea of the construction in the general setting. Fix $\e>0$ and $c\geq 1$. The discretization procedure used in the proof of the estimate \eqref{Hardy} (see the end of Section 3) combined with the above example immediately gives us a probability space and an $A_{\vec{P}}$ weight such that \eqref{rever} holds. Now, in the general case, we embed the weight and the extremal functions into this special context.

\subsection{Sharpness of \eqref{mainin} for general spaces}

It is convenient to split the construction into a few parts. We may and do assume that $c>1$: when $c=1$, then the vector $\vec{w}$ consists of constant weights and the claim follows easily from the unweighted theory. Let $\alpha=\alpha(p,c)$ be the parameter given by \eqref{defalfa}.

\smallskip

Before we proceed, let us briefly describe the idea behind the construction. To this end, we go back to the proof of \eqref{Hardy} presented at the end of the previous section and the discretization argument used there. We distinguished the probability space $((0,T],|\cdot|/T)$ and the analysis depended heavily on the special intervals of the form $(0,t_j]$ and $(t_{j+1},t_j]$, $j=0,\,1,\,2,\,\ldots$. This suggests the approach which can be roughly formulated as follows. From the previous subsection, we know the extremal examples $\vec{f}$, $\vec{w}$ for the multilinear Hardy inequality on $(0,\infty)$, and they should remain remain extremal if we restrict ourselves to the bounded interval $(0,T]$, for some large $T$. Now, we pick small $\delta\in (0,1)$, consider the tree as in the proof of \eqref{Hardy} and consider the conditional expectations $\tilde{f}$, $\tilde{w}$ ($=\tilde{\sigma}^{1-\underline{p}}$) coming from \eqref{condexp}. It is natural to expect that these vectors should be extremal in the above special probability space. For general $(X,\mu)$, the idea is to find first appropriate analogues of the intervals $(0,t_j]$ and then transfer $\tilde{f}$, $\tilde{w}$ to this setting, using appropriate measure-preserving transformation.

\smallskip

Having described the idea, we turn to the rigorous analysis.

\smallskip

\emph{Step 1. Auxiliary geometrical facts and parameters.}
Pick $\tilde{\alpha}<\alpha$. Then the number $\tilde{c}$ given by
$$ \tilde{c}(\tilde{\alpha}+1)\prod_{j=1}^m (1+\tilde{\alpha}(1-p_j'))^{p/p_j'}=1$$
satisfies $\tilde{c}<c$.
Consider the points $\mathcal{P}, \mathcal{Q} \in \R^{m+1}$ given by
$$\mathcal{P} = \left(1,1, \ldots , 1\right) \qquad \mbox{and}\qquad \mathcal{Q} = \left(\frac{1}{1+\tilde{\alpha}(1-p_1')},\ldots,
\frac{1}{1+\tilde{\alpha}(1-p_m')}, \frac{1}{1+\tilde{\alpha}}\right).$$
Then the previous equality is equivalent to saying that $\mathcal{Q}$ lies on the curve $$\lbrace(s_1,s_2,\ldots,s_m,u)\,:\,us_1^{p/p_1'}s_2^{p/p_2'}\ldots s_m^{p/p_m'}=\tilde{c}\rbrace.$$
It is easy to check that the line segment $\mathcal{P}\mathcal{Q}$ is tangent to this curve.
Next, fix $\delta>0$ and let $\mathcal{R} \in \R^{m+1}$ be defined by
$$ \mathcal{R}= \left(\frac{\delta(1-\delta)^{\tilde{\alpha}(1-p_1')}}{1-(1-\delta)^{\tilde{\alpha}(1-p_1')+1}},\ldots,
\frac{\delta(1-\delta)^{\tilde{\alpha}(1-p_m')}}{1-(1-\delta)^{\tilde{\alpha}(1-p_m')+1}}, \frac{\delta(1-\delta)^{\tilde{\alpha}}}{1-(1-\delta)^{\tilde{\alpha}+1}}\right).$$
Clearly, we have $\mathcal{R}\to \mathcal{Q}$ as $\delta\to 0$. Consequently, for $\delta$ sufficiently close to $0$, the
%$$ \mathcal{R}_{m+1}\prod_{j=1}^m (\mathcal{R}_j)^{p/p_j'}\leq c.$$
%Moreover, one easily checks that the line passing through $\mathcal{P}$ and $\mathcal{Q}$ is tangent to the curve $\lbrace(s_1,s_2,\ldots,s_m,u)\,:\,us_1^{p/p_1'}s_2^{p/p_2'}\ldots s_m^{p/p_m'}=\tilde{c}\rbrace$, so for an appropriately small $\delta$, the
 line segment joining $\mathcal{P}$ and $\mathcal{R}$ lies below the curve $\{(s_1,s_2,\ldots,s_m,u)\,:\,us_1^{p/p_1'}s_2^{p/p_2'}\ldots s_m^{p/p_m'}=c\}$. See Figure \ref{PQR}.

 \begin{figure}[htbp]
\begin{center}
\includegraphics[scale=0.65]{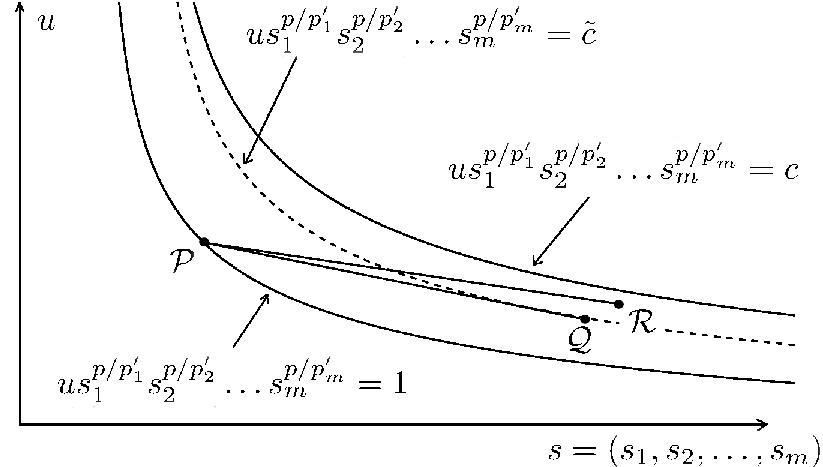}
\caption{The points $\mathcal{P}$, $\mathcal{Q}$, $\mathcal{R} \in \mathbb{R}^{m+1}$. The line segment $\mathcal{P}\mathcal{Q}$ is tangent to the curve $us_1^{p/p_1'}s_2^{p/p_2'}\ldots s_m^{p/p_m'}=\tilde{c}$ and the line segment $\mathcal{PR}$ lies entirely below the curve $us_1^{p/p_1'}s_2^{p/p_2'}\ldots s_m^{p/p_m'}=c$.}\label{PQR}
\end{center}
\end{figure}

\smallskip

\emph{Step 2. Construction.}
Now, recall the following simple measure-theoretic fact, which can be found in \cite{Me}.

\begin{lemma}\label{Melas}
For every $Q \in \mathcal{T}$ and every $\beta \in (0,1)$ there is a subfamily $F(Q) \subset \mathcal{T}$ consisting of pairwise
disjoint subsets of $Q$ such that
$$ \mu\left(\bigcup_{R\in F(Q)} R\right)=\sum_{R\in F(Q)} \mu(R)=\beta\mu(Q).$$
%The elements of $F(Q)$ are called the \emph{atoms} of $\bigcup_{R\in F(Q)}R$.
\end{lemma}

We use this fact recursively and construct an appropriate sequence $A_0\supset A_1\supset A_2\supset \ldots$ of subsets of $X$, which will play the role of the intervals $(0,t_0]$, $(0,t_1]$, $(0,t_2]$, $\ldots$ discussed above. The starting point is the choice $A_0=X$. To describe the inductive step, assume we have constructed $A_n$, which is a union of pairwise almost disjoint elements of $\mathcal{T}$, called the \emph{atoms} of $A_n$. Of course,  this condition is satisfied for $n=0$: we have $A_0=X\in \mathcal{T}$. Then, for each atom $Q$ of $A_n$, we apply the above lemma with $\beta=1-\delta$ and get the corresponding subfamily $F(Q)$. Put $A_{n+1}=\bigcup_Q\bigcup_{Q'\in  F(Q)}Q'$, the first union taken over all atoms $Q$ of $A_n$. Directly from the definition, this set is a union of the family $\{F(Q):Q\mbox{ an atom of }A_n\}$, which consists of pairwise disjoint elements of $\mathcal{T}$. We call these elements the atoms of $A_{n+1}$ and conclude the description of the induction step.

As an immediate consequence of the above construction, we see that if $Q$ is an atom of $A_k$, then for any $n\geq k$ we have
$$ \mu(Q\cap A_n)=\mu(Q)\left(1-\delta\right)^{n-k}$$
and hence
\begin{equation}\label{ufo}
 \mu(Q\cap (A_n\setminus A_{n+1}))=\mu(Q)\left(1-\delta\right)^{n-k}\delta.
 \end{equation}
 In particular, we deduce that for any $n=0,\,1,\,2,\,\ldots$,
 \begin{equation}\label{An}
 \mu(A_n\setminus A_{n+1})=(1-\delta)^n\delta.
 \end{equation}
Now, introduce the weights $w_1=w_2=\ldots=w_m$ on $X$ by the formula
$$ w_j=\sum_{n=0}^\infty \chi_{A_n\setminus A_{n+1}}(1-\delta)^{n\tilde{\alpha}}.$$
and for $j=1,\,2,\,\ldots,\,m$, we let
$$ f_j=\sum_{n=0}^\infty \chi_{A_n\setminus A_{n+1}}(1-r\delta)^{-n(\tilde{\alpha}+1)/p_j},$$
where $r \in (0,1)$ is a fixed parameter. It is instructive to compare the functions $w_j$ and $f_j$ just defined above to their ``continuous'' counterparts \eqref{weiHardy} and \eqref{funHardy}: we immediately see the similarity of these objects, at least when $\tilde{\alpha}$ is close to $\alpha$, $r$ is close to $1$ and $\delta$ is close to $0$.

\smallskip

\emph{Step 3. We have $[\vec{w}]_{A_{\vec{P}}}\leq c$.} Since $w_1=w_2=\ldots=w_m$, we have $v_{\vec{w}}=w_1$ and we need to show that for any $R\in \mathcal{T}$,
$$ \langle w_1\rangle_R \prod_{j=1}^m \langle w_1^{1-p_j'}\rangle_R^{p/p_j'}\leq c.$$
 To prove this, we use \eqref{ufo} to obtain that for each atom $Q$ of $A_k$ we have
\begin{equation}\label{atom0}
\begin{split}
\langle w_1\rangle_Q&=\frac{1}{\mu(Q)}\sum_{n\geq 0} \mu(Q\cap (A_n\setminus A_{n+1}))(1-\delta)^{n\tilde{\alpha}}\\
 &= \sum_{n \geq k} (1-\delta)^{n-k}\delta (1-\delta)^{n\tilde{\alpha}} = \frac{\delta (1-\delta)^{k\tilde{\alpha}}}{1-(1-\delta)^{\tilde{\alpha}+1}}
\end{split}
 \end{equation}
and
\begin{align*}
\langle w_1^{1-p_j'}\rangle_Q = \sum_{n\geq k}(1-\delta)^{n-k}\delta (1-\delta)^{n\tilde{\alpha}(1-p_j')} = \frac{\delta (1-\delta)^{k\tilde{\alpha}(1-p_j')}}{1-(1-\delta)^{\tilde{\alpha}(1-p_j')+1}}.
 \end{align*}
Now, suppose that $R$ is an arbitrary element of $\mathcal{T}$. Then there is an integer $k$ such that $R\subseteq A_{k-1}$ and $R\not\subseteq A_k$. We have
\begin{align*}
 \langle w_1\rangle_R&=\frac{1}{\mu(R)}\int_{R\setminus A_k}w_1\mbox{d}\mu+\frac{1}{\mu(R)}\int_{R\cap A_k}w_1\mbox{d}\mu\\
&=\frac{1}{\mu(R)}\int_{R\setminus A_k}(1-\delta)^{(k-1)\tilde{\alpha}}\mbox{d}\mu+\frac{1}{\mu(R)}\int_{R\cap A_k}w_1\mbox{d}\mu.
\end{align*}
By \eqref{atom0}, applied to each atom $Q$ of $A_k$ contained in $R$, we get
$$ \int_{R\cap A_k}w_1\mbox{d}\mu=\frac{\mu(R\cap A_k)\delta (1-\delta)^{k\tilde{\alpha}}}{1-(1-\delta)^{\tilde{\alpha}+1}}$$
and hence, setting $\eta:=\mu(R\cap A_k)/\mu(R)$, we rewrite the preceding equality as
$$  \langle w_1\rangle_R=(1-\eta)(1-\delta)^{(k-1)\tilde{\alpha}}+\eta\cdot \frac{\delta(1-\delta)^{k\tilde{\alpha}}}{1-(1-\delta)^{\tilde{\alpha}+1}}.$$
Similarly, we obtain
\begin{align*}
& \langle w_1^{1-p_j'}\rangle_R=(1-\eta)(1-\delta)^{(k-1)\tilde{\alpha}(1-p_j')}+\eta \cdot \frac{\delta (1-\delta)^{k\tilde{\alpha}(1-p_j')}}{1-(1-\delta)^{\tilde{\alpha}(1-p_j')+1}},
\end{align*}
which implies
\begin{align*}
& \langle w_1\rangle_R \prod_{j=1}^m \langle w_1^{1-p_j'}\rangle_R^{p/p_j'}\\
&=\bigg((1-\eta)+ \eta \cdot \frac{\delta(1-\delta)^{\tilde{\alpha}}}{1-(1-\delta)^{\tilde{\alpha}+1}}\bigg)\prod_{j=1}^m\bigg((1-\eta) + \eta\cdot \frac{\delta(1-\delta)^{\tilde{\alpha}(1-p_j')}}{1-(1-\delta)^{\tilde{\alpha}(1-p_j')+1}}\bigg)^{\frac{p}{p_j'}}.
\end{align*}
This number does not exceed $c$. Indeed, the right-hand side can be rewritten as
$$ \Big((1-\eta)\mathcal{P}_{m+1}+\eta\mathcal{R}_{m+1}\Big)\prod_{j=1}^m\Big((1-\eta)\mathcal{P}_{j} + \eta\mathcal{R}_{j}\Big)^{\frac{p}{p_j'}}$$
where $\mathcal{P}_n$ and $\mathcal{R}_n$ are the coordinates of the points $\mathcal{P}$ and $\mathcal{R}$. As $\eta$ ranges from $0$ to $1$, the point $(1-\eta)\mathcal{P} + \eta\mathcal{R}$ runs over the line segment $\mathcal{P}\mathcal{R}$ which is entirely contained in $\{(s_1,s_2,\ldots,s_m,u)\,:\,us_1^{p/p_1'}s_2^{p/p_2'}\ldots s_m^{p/p_m'}\leq c\}$ (see Step 1). Since $R$ was arbitrary, we obtain the desired $A_{\vec{P}}$ condition: $[\vec{w}]_{A_{\vec{P}}}\leq c$.

\smallskip

\emph{Step 4. Completion of the proof.} In the same manner as above, one verifies that if $Q$ is an atom of $A_k$, then
\begin{align*}
\langle f_j\rangle_Q =\sum_{n\geq k} (1-\delta)^{n-k}\delta(1-r\delta)^{-n(\tilde{\alpha}+1)/p_j}   = \frac{\delta ( 1-r\delta)^{-k(\tilde{\alpha}+1)/p_j}}{1-(1-\delta)(1-r\delta)^{-(\tilde{\alpha}+1)/p_j}}.
\end{align*}
Note that the ratio of the above geometric series is equal to
$$ (1-\delta)(1-r\delta)^{-(\tilde{\alpha}+1)/p_j} = 1 +\delta\left(\frac{r(\tilde{\alpha}+1)}{p_j} - 1\right) + o(\delta) < 1$$
for sufficiently small $\delta$, since $\tilde{\alpha} < \alpha \leq p_j-1$, so the series is convergent. Hence, on $A_k$ (and hence also on $A_k\setminus A_{k-1}$) we have
$$\mathcal{M}\vec{f}\geq \prod_{j=1}^m \langle f_j\rangle_Q=\prod_{j=1}^m \frac{\delta ( 1-r\delta)^{-k(\tilde{\alpha}+1)/p_j}}{1-(1-\delta)(1-r\delta)^{-(\tilde{\alpha}+1)/p_j}}$$
and therefore, using \eqref{An}, we get
\begin{align*}
\|\mathcal{M}\vec{f}\|_{L^p(v_{\vec{w}})}&\geq \left(\sum_{n\geq 0}\left(\prod_{j=1}^m \frac{\delta ( 1-r\delta)^{-n(\tilde{\alpha}+1)/p_j}}{1-(1-\delta)(1-r\delta)^{-(\tilde{\alpha}+1)/p_j}}\right)^p (1-\delta)^n\delta(1-\delta)^{n\tilde{\alpha}}\right)^{1/p}.
\end{align*}
The ratio of the geometric series in the parentheses is given by
\begin{align*}
 \prod_{j=1}^m (1-r\delta)^{-(\tilde{\alpha}+1)p/p_j}\cdot (1-\delta)^{\tilde{\alpha}+1} &= (1-r\delta)^{-(\tilde{\alpha}+1)}(1-\delta)^{\tilde{\alpha}+1}\\
&=1+\delta\left(r(\tilde{\alpha}+1)-(\tilde{\alpha}+1)\right) + o(\delta)\\
&=1+\delta(\tilde{\alpha}+1)(r-1) + o(\delta) < 1
\end{align*}
as $\delta\to 0$. Thus, the series is convergent, and we obtain
\begin{align*}
 \|\mathcal{M}\vec{f}\|_{L^p(v_{\vec{w}})}&\geq \prod_{j=1}^m \frac{\delta }{1-(1-\delta)(1-r\delta)^{-(\tilde{\alpha}+1)/p_j}}\cdot\left(\frac{\delta}{1- (1-\delta)^{\tilde{\alpha}+1}(1-r\delta)^{-(\tilde{\alpha}+1)}}\right)^{\frac{1}{p}}.
\end{align*}
On the other hand, an analogous computation, again based on the equality \eqref{An}, shows that for $j\geq 1$,
\begin{align*}
\|f_j\|_{L^{p_j}(w_j)}&=\left(\sum_{n\geq 0}(1-r\delta)^{-n(\tilde{\alpha}+1)}(1-\delta)^n\delta(1-\delta)^{n\tilde{\alpha}} \right)^{\frac{1}{p_j}}\\
&= \left(\frac{\delta}{1- (1-\delta)^{\tilde{\alpha}+1}(1-r\delta)^{-(\tilde{\alpha}+1)}}\right)^{\frac{1}{p_j}}.
\end{align*}
Putting the above facts together, we get
$$ \limsup_{\delta \to 0}\frac{\|\mathcal{M}\vec{f}\|_{L^p(v_{\vec{w}})}}{\prod_{j=1}^m \|f_j\|_{L^{p_j}(w_j)}}\geq \prod_{j=1}^m \frac{\delta}{1-(1-\delta)(1-r\delta)^{-(\tilde{\alpha}+1)/p_j}}.$$
Finally, we check that with $\delta$ approaching $0$, the constant on the right converges to
$$ \prod_{j=1}^m \frac{1}{1- \frac{r(\tilde{\alpha}+1)}{p_j}} = \prod_{j=1}^m \frac{p_j}{p_j- r(\tilde{\alpha}+1)}.$$
As $r \to 1, \tilde{\alpha} \to \alpha$, this yields the desired lower bound for the optimal constant.
\medskip

\section{On the search of the Bellman function}

As we have seen above, the key ingredient of the proof of our main estimate \eqref{mainin} lies in the discovery of the best constant in the testing condition \eqref{testing}. The purpose of this section is to sketch an informal reasoning which has led us to this constant and the formula for the special function associated with this estimate.

Let $c\geq 1$ be a given number, $m$ be a positive integer and $\vec{P}=(p_1,p_2,\ldots,p_m)$ be a fixed vector of exponents satisfying $p^{-1}=p_1^{-1}+p_2^{-1}+\ldots+p_m^{-1}$ and $p_j>1$ for each $j$. Suppose that  the underlying probability space $(X, \mu)$ is the interval $[0,1]$ with Lebesgue's measure, and equip it with the tree of dyadic subintervals. Assume further that we are interested in the least constant $\kappa=\kappa_{\vec{P},c}$ such that whenever $R$ is a dyadic subinterval of $[0,1]$ and $\vec{w}$ is a multilinear weight on $[0,1]$ satisfying $[\vec{w}]_{A_{\vec{P}}}=c$, then
\begin{equation}\label{testing2}
\int_R \mathcal{M}(\vec{\sigma}\chi_R)^pv_{\vec{w}}\mbox{d}\mu\leq \kappa^p\int_R \prod_{j=1}^m \sigma_j^{p/p_j}\mbox{d}\mu.
\end{equation}
We start with the description of the general approach.
The abstract Bellman function $\mathbb B$ associated with the above estimate is given by
$$ \mathbb{B}(s;t,u)=\sup\left\{\frac{1}{\mu(R)}\int_R \bigg(\max\left\{\mathcal{M}(\vec{\sigma}\chi_R),t\right\}^pv_{\vec{w}}- \kappa^p \prod_{j=1}^m \sigma_j^{p/p_j}\bigg)\mbox{d}\mu\right\},$$
where the supremum is taken over all $A_p$ weights $\vec{w}$ on $R$ satisfying $\langle \vec{\sigma}\rangle_R=s$, $\langle v_{\vec{w}}\rangle_R=u$  and $[\vec{w}]_{A_{\vec{P}}}\leq c$. At the first glance, the function $\mathbb B$ depends also on the interval $R$, however, this is not the case.  Indeed, for any two dyadic intervals $R_1$ and $R_2$, an affine mapping of one interval onto another puts the $A_{\vec{P}}$ weights in one-to-one correspondence, and such a change of the variable preserves the averages and the $A_{\vec{P}}$ characteristics. On the other hand, considering different domains $R$ is crucial for the understanding of the properties of $\mathbb B$. Another important remark concerns the domain $\mathcal{D}_{\vec{P},c}$ of $\mathbb{B}$: it is equal to
$$ \mathcal{D}_{\vec{P},c}=\Big\{(s;t,u)\in (0,\infty)^m\times (0,\infty)\times (0,\infty)\,:  s^{-p/\underline{p}'}\leq u\leq cs^{-p/\underline{p}'}\Big\}.$$
Indeed, if $u$, $s$ satisfy $u< s^{-p/\underline{p}'}$, then the weight $\vec{w}$ with the above properties does not exist, by the H\"older inequality. Furthermore, if $u> cs^{-p/\underline{p}'}$, then Muckenhoupt's condition is violated at the interval $R$.

Let us study the structure of $\mathbb{B}$. By the validity of \eqref{testing2} (with some a priori unknown constant $\kappa_{\vec{P},c}$), it follows that
\begin{equation}\label{init0}
\mathbb{B}(s;s_1s_2\ldots s_m,u)\leq 0.
\end{equation}
Indeed: we have $\mathcal{M}(\vec{\sigma}\chi_R)\geq \langle \sigma_1\rangle_R\langle \sigma_2\rangle_R\ldots \langle \sigma_m\rangle_R=s_1s_2\ldots s_m$ on $R$, so for any weight $\vec{w}$ as in the definition of $\mathbb{B}(s;s_1s_2\ldots s_m,u)$ we have
$$ \int_R \bigg(\max\left\{\mathcal{M}(\vec{\sigma}\chi_R),s_1s_2\ldots s_m\right\}^pv_{\vec{w}}- \kappa^p \prod_{j=1}^m \sigma_j^{p/p_j}\bigg)\mbox{d}\mu\leq 0.$$
Taking the supremum over all $\vec{w}$, we obtain the estimate \eqref{init0}. The pointwise bound $\mathcal{M}(\vec{\sigma}\chi_R)\geq s_1s_2\ldots s_m$ yields another property of $\mathbb{B}$: we have $\max\left\{\mathcal{M}(\vec{\sigma}\chi_R),t\right\}=\max\left\{\mathcal{M}(\vec{\sigma}\chi_R),s_1s_2\ldots s_m,t\right\}$ and hence
\begin{equation}\label{maxxB}
\mathbb{B}(s;t,u)=\mathbb{B}\big(s;\max\{t,s_1s_2\ldots s_m\},u\big).
\end{equation}
Next, note that  $ \max\left\{\mathcal{M}(\vec{\sigma}\chi_R),t\right\}\geq t$ and
$$ \frac{1}{\mu(R)}\int_R \prod_{j=1}^m \sigma_j^{p/p_j}\mbox{d}\mu\leq \prod_{j=1}^m \langle \sigma_j\rangle_R^{p/p_j}=s^{p/\underline p},$$
and hence we have the majorization
\begin{equation}\label{majj0}
 \mathbb{B}(s;t,u)\geq t^pu-\kappa^ps^{p/\underline{p}}\qquad \mbox{for all }(s;t,u)\in \mathcal{D}_{\vec{P},c}.
\end{equation}
Note that if $u=s^{-p/\underline{p}'}$, then the equality holds: there is only one, constant weight satisfying the prescribed integral requirements.
We move to another property of $\mathbb B$, which is expressed as the following concavity-type condition. Pick $(s;t,u)\in \mathcal{D}_{\vec{P},c}$ such that $t\geq s_1s_2\ldots s_m$, and another two points $(s_-;u_-)$, $(s_+;u_+)\in (0,\infty)^m\times (0,\infty)$ such that
$$ s_\pm ^{-p/\underline{p}'}\leq u_\pm \leq cs_\pm ^{-p/\underline{p}'}.$$
Then we have
\begin{equation}\label{conv00}
\mathbb{B}(s;t,u)\geq \frac{1}{2}\big(\mathbb{B}(s_-;t,u_-)+\mathbb{B}(s_+;t,u_+)\big).
\end{equation}
To see this, split $R$ into its left half $R_-$ and its right half $R_+$, and then pick an $m$-linear weight $\vec{w}_-$ on $R_-$ and an $m$-linear weight $\vec{w}_+$ on $R_+$ satisfying the properties as in the definitions of $\mathbb{B}(s_-;t,u_-)$ and $\mathbb{B}(s_+;t,u_+)$, respectively. Let us splice these two weights into one $m$-linear $w$ on $R$, given by $\vec{w}=\vec{w}_-\chi_{R_-}+\vec{w}_+\chi_{R_+}$. Then we have $[\vec{w}]_{A_{\vec{P}}}\leq c$, $\langle \vec{\sigma}\rangle_R=(\langle \vec{\sigma}_-\rangle_{R_-}+\langle \vec{\sigma}_+\rangle_{R_+})/2=(s_-+s_+)/2$ and similarly $\langle v_{\vec{w}}\rangle_R=(\langle v_{\vec{w}_-}\rangle_{R_-}+\langle v_{\vec{w}_+}\rangle_{R_+})/2=(u_-+u_+)/2$ . Furthermore, since $t\geq s_1s_2\ldots s_m$, we have the identity
$$\max\left\{\mathcal{M}(\vec{\sigma}\chi_R),t\right\}=\begin{cases}
\max\left\{\mathcal{M}(\vec{\sigma_-}\chi_{R_-}),t\right\} & \mbox{on }R_-,\\
\max\left\{\mathcal{M}(\vec{\sigma_+}\chi_{R_+}),t\right\} & \mbox{on }R_+.
\end{cases}$$
 Therefore, by the very definition of $\mathbb{B}$, we may write
\begin{align*}
\mathbb{B}(s;t,u)&\geq \frac{1}{\mu(R)}\int_R \bigg(\max\left\{\mathcal{M}(\vec{\sigma}\chi_R),t\right\}^pv_{\vec{w}}- \kappa^p \prod_{j=1}^m \sigma_j^{p/p_j}\bigg)\mbox{d}\mu\\
&=\frac{1}{2}\Bigg[\frac{1}{\mu(R_-)}\int_{R_-} \bigg(\max\left\{\mathcal{M}(\vec{\sigma_-}\chi_{R_-}),t\right\}^pv_{\vec{w_-}}- \kappa^p \prod_{j=1}^m \sigma_{-j}^{p/p_j}\bigg)\mbox{d}\mu\\
&\qquad +\frac{1}{\mu(R_+)}\int_{R_+} \bigg(\max\left\{\mathcal{M}(\vec{\sigma_-}\chi_{R_+}),t\right\}^pv_{\vec{w_+}}- \kappa^p \prod_{j=1}^m \sigma_{+j}^{p/p_j}\bigg)\mbox{d}\mu\Bigg].
\end{align*}
Taking the supremum over $\vec{w}_\pm$ on the right, we obtain \eqref{conv00}.

Finally, we mention that $\mathbb B$ satisfies a number of homogeneity conditions. Note that if $\vec{w}$ is an $A_{\vec{P}}$-weight and multiply its $j$-th coordinate (for some fixed $j$) by a parameter $\lambda>0$, then the obtained weight $\vec{w}^j=(w_1,w_2,\ldots,w_{j-1},\lambda w_j,w_{j+1},\ldots,w_m)$ satisfies $[\vec{w}^j]_{A_{\vec{P}}}=[\vec{w}]_{A_{\vec{P}}}$ and the integral averages dilate appropriately. Thus by the very definition of  $\mathbb B$,
\begin{equation}\label{homog}
 \mathbb B(s_1,s_2,\ldots,s_{j-1},\lambda s_j,s_{j+1},\ldots,s_m; \lambda t, \lambda^{-p/p_j'}u)=\lambda^{p/p_j}\mathbb{B}(s;t,u).
\end{equation}
The important observation in the reverse direction is the following. Namely, if we take some constant $\kappa_{\vec{P},c}$ and manage to construct \emph{some} function $B$ which satisfies the conditions \eqref{init0}, \eqref{maxxB}, \eqref{majj0} and \eqref{conv00}, then the desired estimate \eqref{testing2} holds true (with the constant $\kappa_{\vec{P},c}$ we have just picked). This can be easily shown by the same reasoning as in Section 3: the conditions \eqref{maxxB} and \eqref{conv00} imply the monotonicity of the sequence $\left(\int_R {B}(x_n;y_n,z_n)\mbox{d}\mu\right)_n$. Furthermore, if convenient, we need not restrict ourselves to $\mathcal{D}_{\vec{P},c}$ and may work with functions $B$ which are defined on larger domains. We would like to emphasize that the homogeneity property is not needed for the reverse implication, but it plays a key role during the search.

So, let us suppose that we search for a function $B:(0,\infty)^m\times (0,\infty)\times (0,\infty)\to \R$ such that the appropriate four conditions hold (in the literature, such a sufficient function $B$ is called a supersolution). Let us also assume that $B$ enjoys the above homogeneity \eqref{homog}. This search will involve a number of guesses and a lot of experimentation. As the base, we can inspect the papers \cite{Os} and \cite{Oss}, in which the related simpler cases (the linear and special multilinear context $p_1=p_2=\ldots=p_m$) have been handled successfully. The general idea is to take a look at the conditions \eqref{init0}, \eqref{maxxB}, \eqref{majj0} and \eqref{conv00}, and \emph{assume} that the inequalities appearing there actually become equalities at some particular instances: these equalities, in turn, lead to the appropriate Bellman function.

\smallskip

For the clarity of the presentation, we have decided to split the further search into several separate steps. We would like to emphasize once again that the reasoning will be very informal at some steps and we will make several guesses about the Bellman function. Furthermore, we will not worry about the regularity of this object and study its derivatives without checking that they actually exist. This is permitted: the primary goal of this section is to provide the idea behind the construction of the Bellman function, the rigorous analysis has been already carried out above.

\smallskip

\emph{Step 1. A special form of the Bellman function.} We start with the concavity property \eqref{conv00}, which in particular implies that an appropriate Hessian matrix (in the variables $s$ and $u$) is seminegative-definite. The degeneration of this condition forces the determinant of the Hessian to vanish, which leads to the corresponding second-order differential equation of Monge--Amp\'ere type. This equation is extremely difficult to solve, and hence other approach is needed. Instead, we {guess} the general ``shape'' of the formula for $B$. In the linear setting \cite{Os}, it turns out that one may take a function $B$ which, for a fixed $t$, is \emph{linear} in $s$ and $u$. In \cite{Oss}, in the special case $p_1=p_2=\ldots=p_m$, an ``almost linear'' function $B$ was considered: for a fixed $t$, it was linear in $u$ and depended on $s$ through $(s_1s_2\ldots s_m)^{1/m}$ (which, in a sense, is very close to linear). In the general context studied in this paper, the exponents $p_1$, $p_2$, $\ldots$, $p_m$ are different, and it is not clear what the analogue of $(s_1s_2\ldots s_m)^{1/m}$ is; actually, it is not even clear if any modification of this expression would work. The key idea, which pushed our analysis forward, is to try $B$ defined as the infimum of some functions which are linear in $s$ and $u$. That is, suppose that for each $t$ we have a family $\Gamma(t)$ which consists of some linear functions $b:[0,\infty)^{m+1}\to \R$ and let
\begin{equation}\label{Bpar}
 B(s;t,u)=\inf\Big\{ b(s;u)\,:b\in \Gamma(t)\Big\}.
\end{equation}
Clearly, for such functions the concavity condition \eqref{conv00} is given for free. Note that for this choice of $B$, the majorization \eqref{majj0} is equivalent to saying that
\begin{equation}\label{majbb}
 b(s;u)\geq t^pu-\kappa^p s^{p/\underline{p}}
\end{equation}
for all $(s;t,u)\in \mathcal{D}_{\vec{P},c}$ and all $b\in \Gamma(t)$.
It remains to guess the family $\Gamma(t)$ and the formula for the linear functions. To this end, fix $\mathbf{b}\in \Gamma(\kappa)$. We will show that by homogeneity, $\mathbf{b}$ induces a whole family of linear functions. Namely, fix an arbitrary $t>0$ and let $s^*=(s^*_1,s^*_2,\ldots,s^*_m)$ be a sequence of positive numbers with $\kappa s_1^*s_2^*\ldots s_m^*=t$. This sequence gives rise to the function
\begin{equation}\label{b*}
 \tilde{b}^{(s^*)}(s;t,u)=(s^*)^{p/\underline{p}}\,\mathbf{b}\left(\frac{s}{s^*};
u(s^*)^{p/\underline{p}'}\right),
\end{equation}
where for brevity, we have used the notation $s/s^*=(s_1/s_1^*,s_2/s_2^*,\ldots, s_m/s_m^*)$.
By \eqref{majbb}, $\tilde{b}^{(s^*)}$ satisfies
\begin{align*}
 \tilde{b}^{(s^*)}(s;t,u)&=(s^*)^{p/\underline{p}}\,\mathbf{b}\left(\frac{s}{s^*};u(s^*)^{p/\underline{p}'}\right)\\
 &\geq (s^*)^{p/\underline{p}}\left(\frac{t^p}{(s_1^*s_2^*\ldots s_m^*)^p}\frac{u}{(s^*)^{-p/\underline{p}'}}-\kappa^p\frac{s^{p/\underline{p}}}{(s^*)^{p/\underline{p}}}\right)=t^pu-\kappa^ps^{p/\underline{p}}.
 \end{align*}
Furthermore, $\tilde{b}^{(s^*)}(\cdot;t,\cdot)$ is obviously linear in $s$ and $u$. This strongly suggests to take
$$ \Gamma(t)=\Big\{\tilde{b}^{(s^*)}(\cdot;t,\cdot)\,:\,\kappa s_1^*s_2^*\ldots s_m^*=t\Big\},$$
and hence, from now on, we search for the special function of the form
\begin{equation}\label{defBB}
 B(s;t,u)=\inf\left\{(s^*)^{p/\underline{p}}\,\mathbf{b}\left(\frac{s}{s^*};
u(s^*)^{p/\underline{p}'}\right)\,:\,\kappa s_1^*s_2^*\ldots s_m^*=t\right\}.
\end{equation}
Note that the `mother'' function $\mathbf{b}$ inducing the function $B$ is not unique: we could as well replace it with $\tilde{b}^{(s^*)}(\cdot;\kappa,\cdot)$ for some arbitrary sequence $s^*$ satisfying $s_1^*s_2^*\ldots s_m^*=1$, and we would obtain the same $B$. This observation will allow us to use some homogeneity-type arguments later on. In the same spirit, note that the function $B$ enjoys the homogeneity property \eqref{homog}.

Before we proceed, let us emphasize that the above class of functions is very restrictive and a priori it is absolutely not clear if it contains an element enjoying the requirements \eqref{init0} -- \eqref{conv00}. Nevertheless, the above assumptions seem reasonable and even if the search was unsuccessful, the analysis would bring some valuable information on the structure of the Bellman function.

We write $\mathbf{b}(s;u)=K\cdot s+Lu+Q$; thus, the problem boils down to the identification of the coefficients $K$, $L$, $Q$ and the constant $\kappa$.

\smallskip

\emph{Step 2. On the coefficients $K$ and $Q$.}
 Let us exploit the majorization condition \eqref{majj0}: it implies $B(s;\kappa,s^{-p/\underline{p}'})\geq \kappa^ps^{-p/\underline{p}'}-\kappa^ps^{p/\underline{p}}$ for all $s$ and hence $\mathbf{b}$ must satisfy
\begin{equation}\label{key_maj}
\mathbf{b}(s;s^{-p/\underline{p}'})\geq \kappa^ps^{-p/\underline{p}'}-\kappa^ps^{p/\underline{p}}
\end{equation}
(simply take $s_1^*=s_2^*=\ldots=s_m^*=1$ in the definition of $B(s;\kappa,s^{-p/\underline{p}'})$).
However, as we discussed above, for points of the form $(s;\kappa,s^{-p/\underline{p}'})$ we have equality in \eqref{majj0}; thus it seems plausible to impose equality in \eqref{key_maj} for some special choice of $s$.
We {assume} that this occurs for some (single) $s$ with $s_1s_2\ldots s_m=1$. Replacing $\mathbf{b}$ with $\tilde{b}^{(s^*)}$ (see the discussion below \eqref{defBB}) if necessary, we may and do assume that this special $s$ equal to $\textbf{1}=(1,1,\ldots,1)$; here we use the fact that the right-hand side of \eqref{key_maj} enjoys an appropriate homogeneity. Thus $K\cdot\textbf{1}+L+Q=0$;   % Then we have
%$$0\leq B(\textbf{1};\kappa,1)\leq \mathbf{b}(\textbf{1};1)=0,$$
%where the first passage is due to \eqref{majj0}. Thus $B(\textbf{1};\kappa,1)= \mathbf{b}(\textbf{1};1)$ and since $B(\cdot,\kappa,\cdot)\leq \mathbf{b}(\cdot,\cdot)$, this in particular implies that the partial derivatives of these functions (with respect to the variables $s$ and $u$) at the point $(\textbf{1};1)$ must match.
furthermore,
%$$ \mathbf b_u(\textbf{1};1)=B_u(\textbf{1};\kappa,1)=L$$
differentiating both sides of \eqref{key_maj} with respect to $s_j$ at the point $(\textbf{1};1)$ yields
$$ \mathbf b_{s_j}(\textbf{1};1)-\frac{p}{p_j'}\mathbf b_u(\textbf{1};1)=-p\kappa^p,$$
or
$$ K_j=\mathbf{b}_{s_j}(\textbf{1};1)=p\left(\frac{L}{p_j'}-\kappa^p\right).$$
Putting all the above facts together, we get
\begin{equation}\label{coming_back}
\begin{split}
 \mathbf{b}(s;u)&=p\sum_{j=1}^m \left(\frac{L}{p_j'}-\kappa^p\right)(s_j-1)+L(u-1)\\
 &=p\sum_{j=1}^m \left(\frac{L}{p_j'}-\kappa^p\right)s_j+Lu+mp(\kappa^p-L).
\end{split}
\end{equation}
Note that for each $j$, the coefficient $K_j=p(L/p_j'-\kappa^p)$ must be nonnegative, since otherwise the inequality \eqref{key_maj} would be violated. Indeed, if we let $s_k=1$ for all $k\neq j$, the inequality is equivalent to
$$ p\left(\frac{L}{p_j'}-\kappa^p\right)s_j+(L-\kappa^p)s_j^{-p/p_j'}+\kappa^p s_j^{p/p_j}+p\kappa^p-pL\left(\frac{1}{p_j'}+\frac1p\right)\geq 0$$
and if we send $s_j$ to $\infty$, we see that the first term is the leading one. Thus $L/p_j'-\kappa^p\geq 0$ for all $j$, in particular we also have
\begin{equation}\label{Lin}
L-\kappa^p\geq 0.
\end{equation}

\smallskip

\emph{Step 3. On the parameter $L$ and the constant $\kappa$.}
To find $L$, we turn our attention to the property \eqref{init0}. In particular, this condition implies that we must have
$$ B(s;\kappa,cs^{-p/\underline{p}'})\leq 0$$
for any $s$ such that $s_1s_2\ldots s_m=\kappa$. By \eqref{defBB}, this is equivalent to saying that
\begin{equation}\label{sayingt}
 \inf\left\{(s^*)^{p/\underline{p}}\,\mathbf{b}\left(\frac{s}{s^*},c\left(\frac{s}{s^*}\right)^{-p/\underline{p}'}\right)\,:\, s_1^*s_2^*\ldots s_m^*=1\right\}\leq 0
 \end{equation}
for any $s$ as previously. We will show that this implies
\begin{equation}\label{alfain} \gamma=\inf\left\{ \mathbf{b}(s;cs^{-p/\underline{p}'})\,:\,s_1s_2\ldots s_m=\kappa\right\}\leq 0.
\end{equation}
To this end, suppose conversely that $\gamma>0$. Directly from \eqref{sayingt}, for any $\e>0$ there is $s^*$ with $s_1^*s_2^*\ldots s_m^*=1$ such that
$$(s^*)^{p/\underline{p}}\,\gamma\leq (s^*)^{p/\underline{p}}\,\mathbf{b}\left(\frac{s}{s^*},c\left(\frac{s}{s^*}\right)^{-p/\underline{p}'}\right)\leq \e.$$
Consequently,
\begin{align*}
 \e&\geq (s^*)^{p/\underline{p}}\,\mathbf{b}\left(\frac{s}{s^*},c\left(\frac{s}{s^*}\right)^{-p/\underline{p}'}\right)\\
 &=(s^*)^{p/\underline{p}}\left[p\sum_{j=1}^m \left(\frac{L}{p_j'}-\kappa^p\right)\frac{s_j}{s_j^*}+mp(\kappa^p-L)\right]+Lcs^{-p/\underline{p}'}\\
 &\geq mp(\kappa^p-L)(s^*)^{p/\underline{p}}+Lcs^{-p/\underline{p}'}\\
&\geq mp(\kappa^p-L)\e\gamma^{-1}+Lcs^{-p/\underline{p}'},
\end{align*}
where in the last passage we have exploited \eqref{Lin}. Note that $Lcs^{-p/\underline{p}'}>0$; thus letting $\e\to 0$, we obtain a contradiction. Thus \eqref{alfain} holds, equivalently,
\begin{equation}\label{infim}
 \inf \left\{p\sum_{j=1}^m \left(\frac{L}{p_j'}-\kappa^p\right)s_j+Lcs^{-p/\underline{p}'}+mp(\kappa^p-L):s_1s_2\ldots s_m=\kappa\right\}\leq 0.
 \end{equation}
Let us find $L$ and $\kappa$ so that we have \emph{equality} here. To this end, we make the convenient substitution $\tilde{L}=L/\kappa^p$ and apply the Young inequality, obtaining
\begin{align*}
p\sum_{j=1}^m \left(\frac{\tilde L}{p_j'}-1\right)s_j+\tilde L cs^{-p/\underline{p}'}
&=\sum_{j=1}^m \frac{p\left(\frac{\tilde L}{p_j'}-1\right)s_jmp_j'}{mp_j'} +\frac{\tilde Lc s^{-p/\underline{p}'} mp}{mp}\\
&\geq mp (\tilde Lc)^{1/mp}\prod_{j=1}^m (\tilde L-p_j')^{1/mp_j'}.
\end{align*}
Hence, if we assume that $\tilde L>1$ satisfies the equation
\begin{equation}\label{defLL}
 (\tilde Lc)^{1/mp}\prod_{j=1}^m (\tilde L-p_j')^{1/mp_j'}=\tilde L-1,
\end{equation}
then reverse to \eqref{infim} holds. Equality occurs for $s$ satisfying
$$ \left(\tilde L-p_j'\right)s_j=\tilde{L}cs^{-p/\underline{p}'}, \qquad j=1,\,2,\,\ldots,\,m.$$
To solve this in $s_j$, we observe that the substitution $s_j=(\tilde L-1)/(\tilde L-p_j')$ transforms the above equality into \eqref{defLL}. To make sure that equality holds in \eqref{infim}, we must assume that the sequence $s$ satisfies $s_1s_2\ldots s_m=\kappa$. This leads to the choice of the optimal constant:
$$ \kappa=\prod_{j=1}^m \frac{\tilde L-1}{\tilde L-p_j'},$$
where $\tilde L$ is given by \eqref{defLL}. Thus we have found all the parameters and thus obtain the candidate for the Bellman function. This is precisely the function we studied in Section 2 above.

\section*{Artificial Intelligence Statement}
No artificial intelligence tools were used to obtain the above result or for editorial purposes.

\end{document}